\documentclass[reqno,11pt]{amsart}
\usepackage[a4paper,bindingoffset=0.5cm,left=2cm,right=2cm,top=2.5cm,bottom=2cm,footskip=.8cm]{geometry}
\usepackage{comment}
\usepackage{mathrsfs}
\usepackage{subcaption} 
\usepackage{tikz}
\usepackage{graphicx} 
\usepackage{amsmath, amssymb, amsthm, amsfonts,amsthm} 
\usepackage[ruled,vlined]{algorithm2e}
\usepackage{hyperref}
\hypersetup{colorlinks=true, linkcolor=blue, citecolor=blue,  filecolor=blue, urlcolor=blue}

\newtheorem{theorem}{Theorem}[section]

\newtheorem{lemma}[theorem]{Lemma}
\newtheorem{remark}[theorem]{Remark}

\numberwithin{equation}{section}

\newcommand{\bs}{\boldsymbol}
\newcommand{\E}{\mathbb{E}}
\newcommand{\HH}[2]{^{(#1,#2)}}
\newcommand{\LL}[2]{_{#1,#2}}
\newcommand{\Hh}[1]{^{(#1)}}
\newcommand{\Ll}[1]{_{#1}}
\newcommand{\LI}{\lambda}
\newcommand{\LG}{\nu}
\newcommand{\GI}{\gamma}
\newcommand{\GG}{\varrho}

\title[Parameter inference for partially observed branching processes]{Parameter inference for \\partially observed branching processes}

\author{Simone Baldassarri}
\address{Gran Sasso Science Institute, Viale Francesco Crispi 7, 67100 L’Aquila, Italy.}
\email{simone.baldassarri@gssi.it}

\author{Michel Mandjes}
\address{Mathematical Institute, Leiden University, P.O. Box 9512, 2300 RA Leiden, The Netherlands.}
\email{m.r.h.mandjes@math.leidenuniv.nl}

\author{Jiesen Wang}
\address{Korteweg-de Vries Institute for Mathematics, University of Amsterdam, Science Park 105-107, 1098 XG Amsterdam, The Netherlands.}
\email{jiesenwang@gmail.com}

\begin{document}

\begin{abstract}
 In this paper, we study an age-dependent branching process. In the simplest setting, the population is divided into two age groups, namely juveniles and adults. Our objective is to estimate the model parameters using observations of the total population size only (i.e., juveniles plus adults).
 Focusing on the ergodic regime of the model, we introduce a method-of-moments estimator and establish its asymptotic normality. Several extensions are discussed, including models with more than two age groups.


\medskip
\noindent \textsc{Acknowledgments.} This research was supported by the European Union’s Horizon 2020 research and innovation programme under the Marie Skłodowska-Curie grant agreement no. 945045, and by the NWO Gravitation project NETWORKS under grant agreement no. 024.002.003. SB was further supported through “Gruppo Nazionale per l’Analisi Matematica, la Probabilità e le loro Applicazioni” (GNAMPA-INdAM).

\end{abstract}

\maketitle

\section{Introduction}
In this paper we consider a branching process with two age groups \cite{AN}, namely {\it juveniles} $(X_n)_{n\ge 0}$ and {\it adults} $(Y_n)_{n\ge 0}$, where $X_n$ and $Y_n$ denote the population sizes in this first and second age class at generation $n$, respectively. The main objective of this paper is to estimate the underlying parameters of the process. In many practical situations, however, only the {\it aggregate population} size \[Z_n := X_n + Y_n\] is observed.
This presents a non-standard scenario for modeling and inference, as the observed data is aggregated, and the age structure is hidden. 
Observe that the bivariate process $(X_n,Y_n)_{n\ge 0}$ is Markovian, but that the aggregated process $(Z_n)_{n\ge 0}$ is {\it not}, which poses significant challenges for standard inference techniques.

The proposed estimation procedure is based on the {\it method of moments} \cite{CasellaBerger2002, LehmannCasella1998}. Specifically, we derive explicit expressions for moments of the observable sequence $(Z_n)_{n\ge 0}$ as functions of the underlying (unknown) model parameters. These theoretical moments are then equated to their empirical counterparts computed from the data, and the resulting system of equations is solved to obtain estimators for the unknown parameters.
For clarity of exposition, we first present the approach in the elementary setting with two age groups and immigration into one of them. We subsequently indicate how the methodology extends in a natural way to models with an arbitrary number of age groups, with immigration occurring in any subset of these groups, as well as to settings in which the reproduction mechanism is more general.

\medskip

We now proceed by outlining the main results of the paper. Restricting attention to the ergodic regime, we first introduce our estimator and establish that the associated system of moment equations admits a unique solution, thereby ensuring well-definedness of the procedure.
As a next step, we show that the empirical moment vector is asymptotically multivariate normal as the number of observations $N$ tends to infinity. Building on this result, we invoke the delta-method~\cite{van1998} to deduce that the proposed estimator itself is asymptotically multivariate normal. Finally, we derive explicit expressions for the entries of the corresponding covariance matrix, expressed directly in terms of the model primitives.

The proposed procedure is evaluated through a series of numerical experiments designed to assess its finite-sample performance. The results indicate that the estimator performs accurately across a range of parameter settings.

We also discuss several extensions of the basic framework. In particular, we consider observation schemes in which only one of the age groups is observed. Another extension concerns models with more than two age groups, where each adult group may reproduce and immigration may occur in any of the groups. We argue that the estimation methodology extends in a conceptually straightforward manner to this more general setting, although the resulting expressions and computations become substantially more involved.

\medskip

Various techniques for handling unobserved or latent quantities have been developed in the mathematical statistics literature. In the present setting, however, their application is far from straightforward.

Likelihood-based methods are difficult to implement in this setting. Evaluating the likelihood would require either full observation of the age structure or an explicit characterization of the transition dynamics of the non-Markovian process $(Z_n)_{n\ge 0}$, which depend on the unobserved population composition. This naturally leads to formulations in terms of hidden or partially observed Markov processes \cite{cappe2005}, for which likelihood evaluation and optimization are often analytically intractable and computationally intensive, especially when the state space is large or the time series is short. Numerical implementations also typically rely on simulation-based procedures, which can be sensitive to tuning and initialization.
By contrast, the method-of-moments approach adopted here avoids likelihood evaluation altogether. By exploiting explicit relationships between observable moments of $(Z_n)_{n\ge 0}$ and the underlying parameters, it circumvents the need to reconstruct the latent age structure and yields estimators that are computationally transparent, scalable, and well suited to aggregated population data.

Expectation-Maximization (EM)-type algorithms might appear to offer a principled
alternative \cite{Dempster1977, LehmannCasella1998}, given the latent structure of the model. However, their application
here is also problematic. The E-step requires computing conditional expectations
with respect to the smoothing distribution of the latent process $(X_n,Y_n)_{n\ge 0}$ given
only observations of $(Z_n)_{n\ge 0}$. Because the observed process lacks the Markov property, this smoothing distribution does not admit a simple recursive representation. As a result, EM implementations would necessarily rely on
numerical or simulation-based approximations, leading to high computational cost. Moreover, the aggregation inherent in $(Z_n)_{n\ge 0}$ may induce weak identifiability, causing slow convergence or sensitivity to
initialization.

Simulation-based approaches, such as particle filters or approximate Bayesian computation (ABC), face similar difficulties in this setting. These methods require repeated simulation of the latent age-structured process and careful tuning of algorithmic parameters, and their performance can degrade substantially when inference is based solely on aggregated observations. For a detailed discussion of particle filtering and sequential Monte Carlo methods, see e.g.\ \cite{cappe2007,doucet2001}, and for ABC techniques, see e.g.\ \cite{beaumont2002,sisson2018}. These references highlight both the computational cost and the sensitivity of such methods to tuning and initialization, illustrating why direct likelihood-free approaches, such as the method-of-moments, may be preferable when only aggregated data $(Z_n)_{n\ge 0}$ are available.

By contrast, the method-of-moments approach developed here directly exploits
explicit relationships between the observable moments of $(Z_n)_{n\ge 0}$ and the
underlying model parameters. This strategy avoids likelihood evaluation
altogether and bypasses the need to reconstruct the hidden age structure. As a
result, the proposed estimators are computationally simple, scalable, and robust
to the non-Markovian nature of the observed process, making them particularly
attractive in settings where only aggregate population data are available.

In the literature on structured population models it is often suggested that parameters governing stage- or age-structured population dynamics cannot be uniquely identified when only aggregate population counts are observed. The intuition behind this concern is that different combinations of reproduction, maturation, and immigration parameters may generate probabilistically identical trajectories for the total population size, making it difficult to disentangle the underlying mechanisms without observing the individual stages. In practice, statistical analyses of structured population models typically rely on stage-specific observations; see, for instance, the classical treatment in \cite{Caswell2001}. Related inference problems also arise in the study of branching processes, where the statistical analysis generally assumes that the different types or generations can be observed separately; see \cite{GUT}. The results obtained in this paper show that, for the stochastic juvenile-adult branching model with immigration considered here, the key parameters can in fact be uniquely identified from the time series of the total population size alone. In particular, under the Poisson assumptions for immigration and reproduction, the parameters governing immigration, juvenile maturation, and adult reproduction can be consistently estimated from aggregate observations.

\medskip

In a broader sense, the scope of this paper lies within what could be called the area of `statistics for stochastic processes'. Such problems can be viewed as instances of {\it inverse problems}, in which the objective is to infer underlying model parameters from limited, indirect, or aggregated data. In the statistics and applied probability literature, a wide range of inverse problems has been studied. A classical example is the estimation of the offspring distribution in a Galton-Watson branching process from longitudinal population data; see, for instance, the textbook~\cite{GUT}. 
In epidemiology, one seeks to estimate the parameters of a stochastic SIR model based on observed infection counts over time~\cite{CAU}. Related inverse problems arise in finance~\cite{ACL, DUF} and in queueing theory~\cite{GuptaMandjesRavnerWang2025, HAN, RAV}. More recently, attention has turned to the estimation of random graph dynamics from partial observations. For example, \cite{Man_Wan} considers a dynamic Erd\H{o}s-R\'enyi graph in which edges randomly alternate between presence and absence, with inference based on observed subgraph counts, \cite{HMW} extends the analysis to the dynamic Chung-Lu graph, and \cite{BW} addresses the problem for a voter model evolving on a dynamic Erd\H{o}s-R\'enyi graph.

\medskip

The remainder of the paper is organized as follows. In Section~\ref{sec:mod}, we introduce the age-dependent branching model and present the proposed method-of-moments estimator. Section~\ref{sec:norm} is devoted to the proof of its asymptotic normality. In Section~\ref{sec:ext}, we discuss two extensions of the basic framework, namely alternative observation schemes and models with more than two age groups. Finally, Section~\ref{sec:disc} concludes the paper with a brief discussion of the results and several directions for future research.

\section{Model and estimator}\label{sec:mod}
In the first part of this paper, we consider a setting in which the population consists of two types, namely juveniles and adults. Although this simplified structure may not be entirely realistic, it provides a convenient and transparent framework for illustrating the construction of our estimator and for establishing its asymptotic normality. We emphasize, however, that this restriction is not essential: in Section~\ref{sec:ext} we show that the methodology developed here extends naturally to substantially more general and realistic age-dependent branching processes.

In the specific setting considered in this section and in Section~\ref{sec:norm}, reproduction occurs only among adults. Each adult produces a random number of juveniles (often referred to as `offspring'). Juveniles require one time step to mature into adults, provided they survive; that is, offspring face a positive probability of dying before reaching adulthood. Adults live for exactly one time unit. In addition, at each generation there is an immigration component, introducing new juveniles into the population, thus helping to prevent extinction. The usual independence assumptions are imposed. Finally, we assume a condition ensuring that the population does not explode; this condition is identified below.

\subsection{Model description, observation scheme}
Let $X_n$ denote the number of juveniles and $Y_n$ the number of adults in timeslot $n$, respectively.  
The variable $X_n$ is updated based on an exogenous input $I_n$ and a sum of random contributions from $Y_{n-1}$, while $Y_n$ evolves via a binomial thinning of $X_{n-1}$. This means that, formally, the dynamics are given by
\begin{equation} \label{eq:mod_dyn}
X_{n+1} = I_n + \sum_{i=1}^{Y_n} G_{n,i}, \qquad Y_{n+1} \sim \text{Bin}(X_n, p),
\end{equation}
where $\{I_n\}_{n \ge 0}$ is a sequence of i.i.d.\ input variables that are distributed as the generic random variable $I\in{\mathbb N}_0$, and $\{G_{n,i}\}_{n,i \ge 1}$ are i.i.d.\ random variables distributed as the generic random variable $G\in{\mathbb N}_0$, independent of $\{I_n\}_{n \ge 0}$.  The model dynamics are illustrated in Figure \ref{F1}. 

The objective is to estimate $p$ and the parameters associated with $I$ and $G$ from observations. If $I$ and $G$ belong to one-parameter families with unknown parameters $\LI$ and $\LG$, respectively, and $p$ is an additional unknown parameter, then the parameter vector to be estimated is ${\boldsymbol\theta} \equiv (p, \LI, \LG)$.
The main complication is that we do {\it not} observe the individual components $\{X_n\}_{n \ge 0}$ and $\{Y_n\}_{n \ge 0}$, but only the {\it total population}. That is, the observable process is $\{Z_n\}_{n \ge 0}$, where \[Z_n := X_n + Y_n.\] Importantly, while the pair $(X_n, Y_n)$ is Markovian, the process $\{Z_n\}_{n \ge 0}$ alone is {\it not}.

It is noted that, while in the next section we focus on constructing an estimator for ${\boldsymbol\theta}$ in the case where the {\it total} population is observed, the proposed methodology also applies (with a minor modification) when only the {\it juvenile} population $\{X_n\}_{n \ge 0}$ is observed; see the discussion in Section~\ref{sec:ext}. Remarkably, in that same section we also argue that ${\bs\theta}$ cannot be estimated when only the {\it adult} population $\{Y_n\}_{n \ge 0}$ is observed.

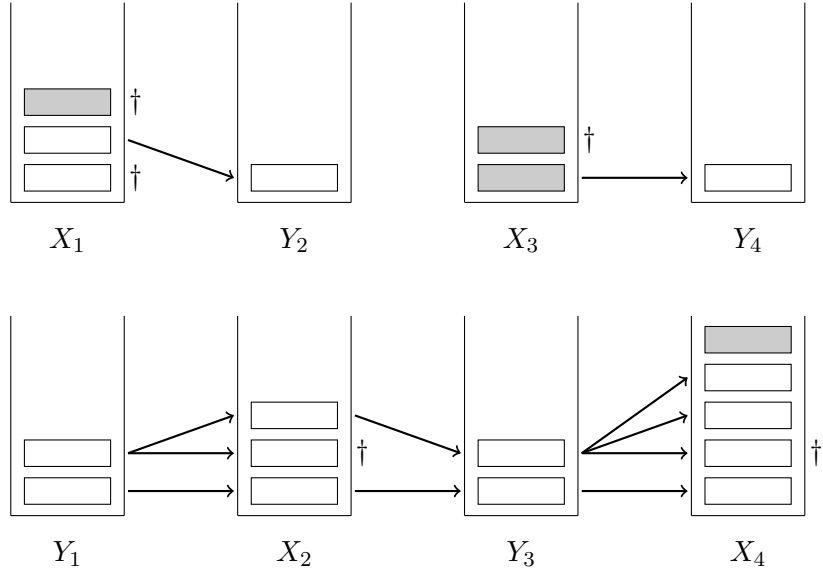
\begin{figure}\begin{tikzpicture}[scale=1]

\def\w{1.5}         
\def\h{1.2}         
\def\H{2.2*\h}      
\def\dx{3}          
\def\dy{1.5}        

\def\hgap{0.18}
\def\vgap{0.15}
\def\smallH{0.35}   

\newcommand{\fillboxes}[2]{%
\pgfmathtruncatemacro{\n}{#1}
\pgfmathtruncatemacro{\shade}{#2}
\foreach \k in {1,...,\n} {
    \pgfmathsetmacro{\y}{\vgap + (\k-1)*(\smallH+\vgap)}
    \ifnum\k>\numexpr\n-\shade\relax
        \filldraw[fill=gray!40] (\hgap,\y) rectangle (\w-\hgap,\y+\smallH);
    \else
        \draw (\hgap,\y) rectangle (\w-\hgap,\y+\smallH);
    \fi
}
}

\foreach \i/\lab/\n/\shade in {0/Y_1/2/0,1/X_2/3/0,2/Y_3/2/0,3/X_4/5/1} {
    \begin{scope}[shift={(\i*\dx,0)}]
        \draw (0,0) -- (0,\H);
        \draw (0,0) -- (\w,0);
        \draw (\w,0) -- (\w,\H);
        \node at (\w/2,-0.5) {$\lab$};
        \fillboxes{\n}{\shade}
    \end{scope}
}

\foreach \i/\lab/\n/\shade in {0/X_1/3/1,1/Y_2/1/0,2/X_3/2/2,3/Y_4/1/0} { 
    \begin{scope}[shift={(\i*\dx,\H+\dy)}]
        \draw (0,0) -- (0,\H);
        \draw (0,0) -- (\w,0);
        \draw (\w,0) -- (\w,\H);
        \node at (\w/2,-0.5) {$\lab$};
        \fillboxes{\n}{\shade}
    \end{scope}
}

\newcommand{\drawTopArrow}[4]{%
\pgfmathsetmacro{\startx}{#1*\dx + \w + 0.05}  
\pgfmathsetmacro{\starty}{\H+\dy + \vgap + (#2-1)*(\smallH+\vgap) + \smallH/2}  
\pgfmathsetmacro{\endx}{#3*\dx - 0.05}  
\pgfmathsetmacro{\endy}{\H+\dy + \vgap + (#4-1)*(\smallH+\vgap) + \smallH/2}  
\coordinate (S) at (\startx,\starty);
\coordinate (E) at (\endx,\endy);
\draw[->, thick] (S) -- (E);
}

\drawTopArrow{0}{2}{1}{1} 
\drawTopArrow{2}{1}{3}{1} 

\newcommand{\drawBottomArrow}[4]{%
\pgfmathsetmacro{\startx}{#1*\dx + \w + 0.05}  
\pgfmathsetmacro{\starty}{\vgap + (#2-1)*(\smallH+\vgap) + \smallH/2}  
\pgfmathsetmacro{\endx}{#3*\dx - 0.05}  
\pgfmathsetmacro{\endy}{\vgap + (#4-1)*(\smallH+\vgap) + \smallH/2}  
\coordinate (S) at (\startx,\starty);
\coordinate (E) at (\endx,\endy);
\draw[->, thick] (S) -- (E);
}

\drawBottomArrow{0}{1}{1}{1} 
\drawBottomArrow{0}{2}{1}{2} 
\drawBottomArrow{0}{2}{1}{3} 
\drawBottomArrow{1}{3}{2}{2} 
\drawBottomArrow{1}{1}{2}{1} 
\drawBottomArrow{2}{1}{3}{1} 
\drawBottomArrow{2}{2}{3}{2} 
\drawBottomArrow{2}{2}{3}{3} 
\drawBottomArrow{2}{2}{3}{4} 

\newcommand{\addDagger}[3]{%
\pgfmathsetmacro{\xD}{#1*\dx + \w + 0.05 + 0.1} 
\pgfmathsetmacro{\yD}{#3 + \vgap + (#2-1)*(\smallH+\vgap) + \smallH/2}
\node at (\xD,\yD) {$\dag$};
}

\addDagger{0}{1}{\H+\dy}  
\addDagger{0}{3}{\H+\dy}  
\addDagger{2}{2}{\H+\dy}  

\addDagger{1}{2}{0}       
\addDagger{3}{2}{0}       

\end{tikzpicture}

\caption{\label{F1}Illustration of the evolution of the populations $X_n$ and $Y_n$ for $n=1,\ldots,4$.  
(a) Grey-shaded individuals represent immigrants in each generation, with counts $I_1=1$, $I_2=0$, $I_3=2$, and $I_4=1$.  
(b) In the boxes $X_n$, a dagger ($\dagger$) next to an individual indicates that it `dies', meaning it does not survive to become a member of $Y_{n+1}$. In contrast, arrows connecting individuals in $X_n$ to $Y_{n+1}$ represent juveniles that successfully mature into adults.  
(c) Arrows from $Y_n$ to $X_{n+1}$ depict the offspring produced by members of $Y_n$. For example, $G_{3,2}=3$ indicates that the second member of $Y_3$ produces three offspring in $X_4$; in our sample path, we also have $G_{1,1}=1$, $G_{1,2}=2$, $G_{2,1}=0$, and $G_{3,1}=1$.
Overall, the figure visualizes the generational dynamics, covering survival, reproduction, and immigration events.}
\end{figure}

The process $\{X_n,Y_n\}_{n\ge 0}$ is clearly Markovian. As we will establish below, the process is ergodic under the stability condition
\begin{equation}
\varrho:= p\,\E[G]<1.
\tag{S}\label{S}
\end{equation}

Throughout this paper, we specifically pay attention to the case that the random variables $I$ and $G$ have Poisson distributions:
\begin{equation}
I\sim {\rm Poisson}(\lambda),\quad G\sim {\rm Poisson}(\nu).
\tag{P}\label{P}
\end{equation}

\subsection{Estimator}
We consider the convenient situation that the process $\{X_n,Y_n\}_{n\ge 0}$ is ergodic, and that we start in stationarity. Suppose that the moments we target are
\[ {\boldsymbol m}\equiv (  m_{1},  m_{2},  m_{12}):=\bigl( \E[Z_n],\ \E[Z_n^2],\ \E[Z_{n+1} Z_n] \bigr),\]
which, by our assumptions, do not depend on $n$. 
Given observations $Z_1, \dots, Z_N$ for some $N\ge 2$, the corresponding empirical moments are $\hat{\boldsymbol m}_N\equiv (\hat m_{1,N},\hat m_{2,N},\hat m_{12,N})$, with
\[
\hat m_{1,N} = \frac{1}{N}\sum_{n=1}^N Z_n, 
\qquad 
\hat m_{2,N} = \frac{1}{N}\sum_{n=1}^N Z_n^2, 
\qquad 
\hat m_{12,N} = \frac{1}{N-1}\sum_{n=1}^{N-1} Z_{n+1} Z_n,
\]
respectively.
The method-of-moments estimators $(\hat p_N, \hat\lambda_{I,N}, \hat\lambda_{G,N})$ are then obtained by solving the following three equations with three unknowns:
\[
\bigl( \E[Z_n],\ \E[Z_n^2],\ \E[Z_{n+1} Z_n] \bigr)
=
\bigl( \hat m_{1,N},\ \hat m_{2,N},\ \hat m_{12,N} \bigr),
\]
where it is observed that the left-hand side consists of functions of entries of ${\bs\theta}$.
Below we demonstrate this approach by detailing the three steps: (a)~deriving expressions 
for the moments, (b)~expressing the moments in terms of the unknown parameters ${\bs\theta}$, (c)~solving the moment equations, thus rendering our estimators. When setting up these estimators, we assume that $I$ and $G$ have finite second moments. 
\begin{itemize}
   \item[(a)] 
We derive expressions for $\E[Z_n]$, $\E[Z_n^2]$, and $\E[Z_{n+1} Z_n]$. Using the identity $Z_n = X_n + Y_n$, we obviously have, for any $n$,
\begin{align}
\E[Z_n] &= \E[X_n] + \E[Y_n], \label{eq:EZ} \\[1ex]
\E[Z_n^2] &= \E[X_n^2] + 2\,\E[X_n Y_n] + \E[Y_n^2], \label{eq:EZ2} \\[1ex]
\E[Z_{n+1} Z_n] &= \E[X_{n+1} X_n] + \E[X_{n+1} Y_n] 
+ \E[Y_{n+1} X_n] + \E[Y_{n+1} Y_n]. \label{eq:EZlag}
\end{align}
Using the model dynamics \eqref{eq:mod_dyn} and Wald's equation, this last expression can be rewritten as
\begin{align}
\E[Z_{n+1} Z_n]
&= \E[I] \,\E[X_n] 
+ \E[G] \,\E[X_n Y_n] 
+ \E[I] \,\E[Y_n] 
+ \E[G]\, \E[Y_n^2] \notag \\
&\quad + p\,\E[X_n Y_n] 
+ p\,\E[X_n^2]. \label{eq:EZlag-expanded}
\end{align}
Consequently, it suffices to derive expressions for the five quantities
\[
\E[X_n], \quad \E[Y_n], \quad \E[X_n Y_n], \quad \E[X_n^2], \quad \E[Y_n^2].
\]
\item[(b)]
{Observe} that the only information about $I$ and $G$ that we need in \eqref{eq:EZlag-expanded} is $\E[I]$ and $\E[G]$. 
It is now straightforward to conclude that, again as an application of Wald's equation, $\E[X_n]=\E[I]+\E[G]\,\E[Y_{n-1}]$ and $\E[Y_n]=p\,\E[X_{n-1}]$, confirming the stationarity condition $\varrho:=p\,\E[G]<1$ as given by \eqref{S}.  Thus, for any $n$,
\begin{align}
\E[X_n] &= \frac{\E[I] }{1 - p\,\E[G]  }, 
\qquad 
\E[Y_n] = p\,\E[X_n]=\frac{p\,\E[I]  }{1 - p\,\E[G]  }; \label{X and Y}
\end{align}
these means \eqref{X and Y} follow by using that in stationarity $\E[X_{n}]= \E[X_{n+1}]$ and $\E[Y_{n}]= \E[ Y_{n+1}]$.
In addition, with ${\mathbb F}[G,I]:=2\E[I]\,\E[G]+{\mathbb V}{\rm ar}[G]$, for any $n$,
\begin{align}
\E[X_n Y_n] &=\frac{p\,\E[I]\,\E[X_n]}{1 - p\,\E[G]  }= \frac{p\,\E[I] ^2 }{(1 - p\,\E[G]  )^2}, \label{XY}\\[1ex]
\left(\begin{array}{c}\E[X_{n}^2]\\
\E[Y_{n}^2]\end{array}\right)&=\left(\begin{array}{cc}1&-(\E[G])^2\\-p^2&1\end{array}\right)^{-1}\left(\begin{array}{c}{\mathbb F}[G,I]\,\E[Y_{n}]+\E[I^2]\\p(1-p)\,\E[X_{n}] \\
\end{array}\right) \\
&=\dfrac{1}{1-p^2\mathbb{E}[G]^2}\left(\begin{array}{cc}1&\E[G]^2\\p^2&1\end{array}\right)\left(\begin{array}{c}{\mathbb F}[G,I]\,\E[Y_{n}]+\E[I^2]\\p(1-p)\,\E[X_{n}] \\
\end{array}\right). \label{Y2}
\end{align}
The expression for $\E[X_n Y_n]$, as presented in \eqref{XY}, is found via a straightforward application of the tower property. Indeed, we observe
\begin{align*}
    \E[X_{n+1} Y_{n+1}]&=\E\left[\E\left[\left(I_{n}+\sum_{i=1}^{Y_n}G_{n,i}\right){\rm Bin}(X_n,p)\Big\vert X_n,Y_n\right]\right]\\
    &= \E\Big[\big(\E[I]+Y_n\,\E[G]\big)\,p\,X_n\Big] = p\,\E[I]\,\E[X_n] + p\,\E[G]\,\E[X_nY_n].
\end{align*}
Then, using that in stationarity $\E[X_{n} Y_{n}]= \E[X_{n+1} Y_{n+1}]$, and recalling the expression for $\E[X_n]$, the above expression for $\E[X_n Y_n]$ directly follows. (As an aside, we observe that our computation reveals that $X_n$ and $Y_n$ are uncorrelated; via some additional computations we verified  that they are {actually \it not} independent.)
Along similar lines \eqref{Y2} was proven; the key intermediate step is to verify that  
\[\left(\begin{array}{c}\E[X_{n+1}^2]\\
\E[Y_{n+1}^2]\end{array}\right)=\left(\begin{array}{c}(\E[G])^2\,\E[Y_{n}^2]\\
p^2\E[X_{n}^2]\end{array}\right)+\left(\begin{array}{c}{\mathbb F}[G,I]\,\E[Y_{n}]\\p(1-p)\,\E[X_{n}]
\end{array}\right)+\left(\begin{array}{c}\E[I^2]\\0\end{array}\right).\]

\noindent
For the `Poisson case', i.e., under \eqref{P}, the above expressions can be evaluated somewhat more explicitly, using that $\LI=\E[I]={\mathbb V}{\rm ar}[I]$ and $\LG=\E[G]={\mathbb V}{\rm ar}[G]$,
\begin{align}
\E[X_n] &= \frac{\LI}{1 - p\,\LG}, 
\qquad 
\E[Y_n] = \frac{\LI p}{1 - p\,\LG},\qquad
\E[X_n Y_n] = \frac{\LI^2 p}{(1 - p\,\LG)^2}, \label{momPois}\\[1ex]
\E[X_n^2] &= \frac{\LI}{1 - p\,\LG} 
+ \frac{\LI^2 + \LI p (\LG^2 + 2 \LI \LG) / (1 - p\,\LG)}
{1 - p^2\,\LG^2 }, \label{momPois1}\\[1ex]
\E[Y_n^2] &= \frac{\LI p}{1 - p\,\LG} 
+ \frac{p^2 \bigl(\LI^2 + \LI p (\LG^2 + 2 \LI \LG) / (1 - p\,\LG)\bigr)}
{1 - p^2\,\LG^2 }.\label{momPois2}
\end{align} 

\item[(c)]
Combining these expressions with \eqref{eq:EZ}--\eqref{eq:EZlag-expanded}, we have succeeded in expressing the moment vector ${\boldsymbol m}\equiv (\E[Z_n],\E[Z_n^2],\E[Z_{n+1} Z_n])$ in terms of the parameter ${\boldsymbol\theta}\equiv(p, \LI, \LG).$ We have thus identified a mapping $F:{\mathbb R}^3\to {\mathbb R}^3$ such that ${\boldsymbol m} = F({\boldsymbol\theta})$. This leads to the estimator 
\[\hat{\boldsymbol\theta}_N:= F^{-1}(\hat{\boldsymbol m}_N).\]
\end{itemize}

We know discuss  the well-definedness of the estimator $\hat{\boldsymbol\theta}_N$.
We consider the parameter space
\[
\Theta= \{ (p,\LI,\LG) : p\in(0,1), \,\E[I]>0,\, \E[G]>0, \,p\,\E[G]<1 \}. 
\]
To properly define $\hat{\boldsymbol\theta}_N$ as $F^{-1}(\hat{\boldsymbol m}_N)$, we would like the following identifiability property to apply:
\begin{equation}
\mbox{For any ${\bs m}\in {\mathbb R}_+^3$, the equation ${\boldsymbol m} = F({\boldsymbol\theta})$ has a unique solution in $\Theta$.}
\tag{U}\label{U}
\end{equation}
The following lemma shows that for Poissonian $I$ and $G$, under the stability condition $\varrho=p\,\nu<1$, there is identifiability. 

\begin{lemma}[Identifiability]\label{lem:uni}
Under assumptions~\eqref{P} and \eqref{S}, the identifiability condition~\eqref{U} holds.
\end{lemma}

{\it Proof.} 
By direct computations, we can write
\begin{align}
\mathbb{E}[Z_n^2] - \mathbb{E}[Z_n]^2 &= \dfrac{\LI(1+p+p(1-p)\LG^2)}{(1-p\LG)^2(1+p\LG)}, \\
\mathbb{E}[Z_{n+1}Z_n] - \mathbb{E}[Z_n]^2 &= \dfrac{p\LI(1+\LG+p(1-p)\LG^2)}{(1-p\LG)^2(1+p\LG)}.
\end{align}
Thus, by letting $R:=(\mathbb{E}[Z_{n+1}Z_n] - \mathbb{E}[Z_n]^2) /(\mathbb{E}[Z_n^2] - \mathbb{E}[Z_n]^2)$,
\[
\dfrac{\mathbb{E}[Z_{n+1}Z_n] - \mathbb{E}[Z_n]^2}{\mathbb{E}[Z_n^2] - \mathbb{E}[Z_n]^2} = \dfrac{p(1+\LG+p(1-p)\LG^2)}{1+p+p(1-p)\LG^2} =: F(p,\LG);
\]
observe that $\lambda$ has canceled.
Note that the function $F(p,\LG)$ is strictly increasing in $p$ for any admissible $\LG>0$, namely, $p\LG<1$. Indeed, by a direct computation, $\partial F/\partial p$ is positive if and only if
\[
\LG[1+2p(1-p)\LG+p^2(p-1)^2\LG^3] + (1-p\LG)(1+p\LG)>0,
\]
which is true for any $\theta\in\Theta$. By arguing in a similar manner, we deduce that the function $F(p,\LG)$ is strictly increasing in $\LG$ for any admissible $p\in(0,1)$. Thus, for any fixed $p$, there exists a unique solution $\LG(p)$ to the equation $F(p,\LG(p))=R$, and from the implicit function theorem we have
\[
\dfrac{\mathrm d \LG(p)}{\mathrm d p} = - \dfrac{\partial F/\partial p}{\partial F/\partial\LG} <0.
\]
This implies that the map $p\mapsto\LG(p)$  is monotonically decreasing. 

Considering the ratio of $\mathbb{E}[Z_n^2] - \mathbb{E}[Z_n]^2$ and $\mathbb{E}[Z_n]$, we observe that again $\lambda$ cancels; from the equation
\[
\dfrac{\mathbb{E}[Z_n^2] - \mathbb{E}[Z_n]^2}{\mathbb{E}[Z_n]} = \dfrac{1+p+p(1-p)\LG(p)^2}{1-p^2\LG(p)^2}=:G(p,\LG(p)),
\]
which is strictly increasing in $p$, we deduce that the intersection between the curve $F(p,\LG(p))=R$ and the function $G(p,\LG(p))$ consists of a unique $p\in(0,1)$. This in turn implies that $\LG$ is uniquely identifiable. Finally, by using $\mathbb{E}[Z_n]=\LI(p+1)/(1-p\LG)$ with $p$ and $\LG$ being identified already, the parameter $\LI$ is also uniquely determined.  \hfill$\Box$

\section{Asymptotic normality}\label{sec:norm}

The main goal of this section is to prove that the estimator $\hat{\boldsymbol\theta}_N$ is asymptotically normal. The approach is to first establish that the estimator $\hat{\boldsymbol m}_N$ is asymptotically normal, and to then apply the delta-method. 

\subsection{Main result} Based on the main result of this section (Theorem \ref{thm:main} below), we have that for large $N$ the estimator $\hat{\bs\theta}_N$ behaves as the unknown parameter ${\bs\theta}$, perturbed by ${\mathscr N}/{\sqrt{N}}$,  where ${\mathscr N}$ is a trivariate centered normal random vector.

\begin{theorem}[Asymptotic normality]  \label{thm:main} Assume that $I$ and $G$ have finite fourth moments, and that the identifiability condition \eqref{U} and  the stability condition \eqref{S} hold. Then the estimator $\hat{\bs\theta}_N$ is asymptotically normal as $N \to \infty$, that is,
\[\sqrt{N}\big(\hat{\bs \theta}_N-\bs\theta\big)\to {\mathscr N}\]
as $N\to \infty$,
with ${\mathscr N}$ denoting a trivariate centered normal random vector with non-degenerate covariance matrix $\Sigma$, {which is given in \eqref{eq:Sigma} below.}
\end{theorem}
The main goal of this subsection is to prove this theorem, conditional on the entries of $\Sigma$ being finite. In the next subsection we show that these entries are indeed finite if the stability condition is in place. 

\medskip

{\it Proof.} The proof starts by observing that, under our stability condition,  $(X_n,Y_n,X_{n+1},Y_{n+1})$ is a stationary Markov chain on the state space ${\mathbb N}_0^4$. By virtue of the central limit
for stationary processes \cite{ibragimov1962}, we have that, for a function $f:{\mathbb N}_0^{4}\to {\mathbb R}$,
\[\sqrt{N}\left(\frac{1}{N}\sum_{n=1}^{N-1}f(X_n,Y_n,X_{n+1},Y_{n+1})- \E[f(X_0,Y_0,X_{1},Y_{1})]\right)\]
converges to a zero-mean normally distributed random variable, with a variance $\sigma^2_f$ that is parameterized by the chosen function $f$, under the condition that the candidate variance is finite. Then we pick, for ${\alpha}_1,\alpha_2,\alpha_3\in{\mathbb R}$,
\[f(x_0,y_0,x_1,y_1):={\alpha}_1\,{\bs 1}^\top \left(\begin{array}{c}x_0\\y_0\end{array}\right)+
{\alpha}_2\,{\bs 1}^\top \left(\begin{array}{c}x_0^2\\2x_0y_0\\y_0^2\end{array}\right)+
{\alpha}_3\,{\bs 1}^\top \left(\begin{array}{c}x_0x_1\\x_0y_1\\y_0x_1\\y_0y_1\end{array}\right).\]
Recalling \eqref{eq:EZ}--\eqref{eq:EZlag},
by applying the {\it Cram\'er-Wold device} \cite{kallenberg2002} we have proven that 
\[\sqrt{N}\left(\left(\begin{array}{c}
     \hat m_{1,N} \\\hat m_{2,N}\\\hat m_{12,N}
\end{array}\right)-\left(\begin{array}{c} \E[Z_0]\\ 
\E[Z_0^2]\\
      \E[Z_0 Z_1]
\end{array}\right)\right)\to {\mathscr N}_\varsigma\]
as $N\to\infty$, with ${\mathscr N}_\varsigma$ denoting a trivariate centered normal random vector with covariance matrix $\Sigma_\varsigma$, and where $\E[Z_0]$, $\E[Z_0^2]$ and $\E[Z_0Z_1]$ directly follow from \eqref{X and Y}--\eqref{Y2}. In the next subsection, we determine the entries of $\Sigma_\varsigma$, in particular showing that they are finite under the condition we imposed, namely that $I$ and $G$ have finite fourth moments. Then by the {\it delta-method} \cite{van1998}, recalling that $\smash{\hat{\boldsymbol\theta}_N:= F^{-1}(\hat{\boldsymbol m}_N)}$, the  statement follows, with
\begin{equation}\label{eq:Sigma}
\Sigma :=  G(\bs\theta)\,\Sigma_\varsigma \,G(\bs\theta)^\top,
\end{equation}
where $G(\bs\theta):=\nabla F^{-1}(\bs\theta)$ is the Jacobian of $F^{-1}(\cdot)$.       
\hfill$\Box$

\subsection{Entries of covariance matrix} \label{covm} In this subsection, we determine the entries of $\Sigma_\varsigma$, showing in the process that they are finite under the condition $\varrho < 1$.

We first introduce some notation. For stochastic vectors ${\bs A}\in{\mathbb R}^{\rm m}$ and ${\bs B}\in{\mathbb R}^n$, we define the objects
\[{\mathbb S}^\circ[{\bs A},{\bs B}]:=\E[{\bs A}\,{\bs B}^\top],\quad{\mathbb S}[{\bs A},{\bs B}]:=\E[{\bs A}\,{\bs B}^\top]-\E[{\bs A}]\,\E[{\bs B}^\top]. \]
Specifically, ${\mathbb S}[{\bs A},{\bs B}]$ is interpreted as the $m\times n$ covariance matrix pertaining to the vector ${\bs A}$ and ${\bs B}$. With slight abuse of notation, we write
\[{\mathbb S}[{\bs A}]:={\mathbb S}[{\bs A},{\bs A}]. \]
The vector ${\bs 1}$ is the all-ones vector (with a dimension that becomes clear from the context). 

\noindent (A)~We start by determining 
\[\varsigma\HH{1}{1}:=\lim_{N\to\infty} N\,{\mathbb V}{\rm ar}[\hat m_{1,N}]={\mathbb V}{\rm ar}[Z_0] +2 \sum_{n=1}^\infty {\mathbb C}{\rm ov}(Z_0,Z_n). \]
As a first step, we provide an expression, with ${\bs V}\Hh{1}_n:=(X_n,Y_n)^\top$, for the conditional mean
\[{\bs M}_n\Hh{1}:= \E\left[{\bs V}_n\Hh{1}\,|\,{\bs V}_0\Hh{1}\right].\] 
Clearly,
\[{\bs M}\Hh{1}_n=\sum_{k,\ell}\E[{\bs V}\Hh{1}_n\,|\,X_{n-1}=k,Y_{n-1}=\ell]\,{\mathbb P}(X_{n-1}=k,Y_{n-1}=\ell\,|\,X_0,Y_0).\]
From \eqref{eq:mod_dyn} we directly see that
$\E[{\bs V}_n\Hh{1}\,|\,X_{n-1}=k,Y_{n-1}=\ell]= (\E[I]+\ell\,\E[G], k\,p)^\top,$
thus obtaining the recursion
\[{\bs M}_n\Hh{1} = \left(\begin{array}{c}\E[I]+\E[Y_{n-1}\,|\,X_0,Y_0]\,\E[G]\\ \E[X_{n-1}\,|\,X_0,Y_0]\,p\end{array}\right)={\bs a}\Ll{1} + D\Ll{1} {\bs M}_{n-1}\Hh{1},\]
where
\[{\bs a}\Ll{1}:= \left(\begin{array}{c} \E[I]\\0
\end{array}\right) ,\quad D\Ll{1}:=\left(\begin{array}{cc}0&\E[G]\\p&0\end{array}\right).\]
This recursion can be solved by iterating $n$ times. We readily find, if the spectral radius $\|D\Ll{1}\|$ of $D\Ll{1}$ is strictly smaller than $1$, or equivalently $\varrho<1$,
\begin{equation}
\notag{\bs M}_n\Hh{1} = \sum_{m=0}^{n-1} D\Ll{1}^{m}{\bs a}\Ll{1} + D\Ll{1}^n {\bs M}_0\Hh{1}=(I-D\Ll{1})^{-1}(I-D\Ll{1}^n)\,{\bs a}\Ll{1} + D\Ll{1}^n {\bs M}_0\Hh{1},\end{equation}
and hence
\[{\bs \mu}\Ll{1}:=\E\left[{\bs V}_0\Hh{1}\right]=(I-D\Ll{1})^{-1}{\bs a}\Ll{1},\]
which is in line with \eqref{X and Y}.
It entails that we can alternatively write (recalling that $(I-D\Ll{1})^{-1}$ and $(I-D\Ll{1}^n)$ commute)
\begin{equation}
\label{Mn} {\bs M}_n\Hh{1} ={\bs \mu}\Ll{1} + D\Ll{1}^n\left({\bs V}_0\Hh{1}-{\bs \mu}\Ll{1}\right),
\end{equation}
using the obvious fact that ${\bs M}_0\Hh{1}={\bs V}_0\Hh{1}$.
It now follows from \eqref{Mn} that
\begin{align*}{\mathbb S}^\circ[{\bs V}_n\Hh{1},{\bs V}_0\Hh{1}]&=(I-D\Ll{1}^n)(I-D\Ll{1})^{-1}\,{\bs a}\Ll{1}{\bs \mu}\Ll{1}^\top + D\Ll{1}^n\,{\mathbb S}^\circ[{\bs V}_0\Hh{1}]={\bs \mu}\Ll{1}{\bs \mu}\Ll{1}^\top + D\Ll{1}^n ({\mathbb S}^\circ[{\bs V}_0\Hh{1}]-{\bs \mu}\Ll{1}{\bs \mu}\Ll{1}^\top),
\end{align*}
from which we conclude that ${\mathbb S}[{\bs V}_n\Hh{1},{\bs V}_0\Hh{1}]$ decays in a matrix-geometric way:
\[{\mathbb S}[{\bs V}_n\Hh{1},{\bs V}_0\Hh{1}]=D\Ll{1}^n\, {\mathbb S}[{\bs V}_0\Hh{1}].\]
This means that 
\[{\mathbb C}{\rm ov}(Z_0,Z_n) = {\bs 1}^\top D\Ll{1}^n\, {\mathbb S}[{\bs V}_0\Hh{1}]\,{\bs 1},\]
so that 
\begin{align*}
   \varsigma\HH{1}{1}&={\bs 1}^\top \,{{\mathbb S}[{\bs V}_n\Hh{1}]}
   \,{\bs 1}+2\cdot{\bs 1}^\top D\Ll{1}(I-D\Ll{1})^{-1}\, {\mathbb S}[{\bs V}_0\Hh{1}]\,{\bs 1}={\bs 1}^\top (I+D\Ll{1})(I-D\Ll{1})^{-1}\, {\mathbb S}[{\bs V}_0\Hh{1}]\,{\bs 1} <\infty,
\end{align*}
where it is observed that ${\mathbb S}[{\bs V}_0\Hh{1}] <\infty$ (entrywise, that is) if $I$ and $G$ have finite second moments (as a direct consequence of the Cauchy-Schwarz inequality). 

(B)~The next step is to study
\[\varsigma\HH{2}{2}:=\lim_{N\to\infty} N\,{\mathbb V}{\rm ar}[\hat m_{2,N}]={\mathbb V}{\rm ar}[Z_0^2] + 2\sum_{n=1}^\infty {\mathbb C}{\rm ov}(Z_0^2,Z_n^2). \]
We follow the same approach as under (A).
We start by defining the object
\[{\bs M}_n\Hh{2}:= \E\left[{\bs V}_n\Hh{2}\,|\,{\bs V}_0\Hh{1}\right],\]
with ${\bs V}\Hh{2}_n:=(X^2_n,X_nY_n, Y_n^2)^\top$.
Observe that, with as before ${\mathbb F}[G,I]=2\E[I]\,\E[G]+{\mathbb V}{\rm ar}[G]$, using \eqref{eq:mod_dyn} and elementary calculation rules,
\begin{align*}
    \E\left[{\bs V}\Hh{2}_n\,\Big|\,X_{n-1}=k,Y_{n-1}=\ell\right]&= \E\left[\left(\begin{array}{c}(I_{n-1}+\sum_{i=1}^\ell G_{n-1,i})^2\\(I_{n-1}+\sum_{i=1}^\ell G_{n-1,i})\,{\rm Bin}(k,p)\\{\rm Bin}^2(k,p)\end{array}\right)\right]\\
    &=\left(\begin{array}{c}\E[I^2]+{\mathbb F}[G,I]\,\ell+(\E[G])^2\ell^2\\p\,\E[I]\,k+p\,\E[G]\,k\ell\\p(1-p)\,k+p^2 \,k^2\end{array}\right).
\end{align*}
This leads to the following recursion, with ${\bs M}\Hh{1}_n$ given in \eqref{Mn}:
\[{\bs M}_n\Hh{2} ={\bs a}\Ll{2} + D\Ll{2} {\bs M}_{n-1}\Hh{2}+ \bar D\Ll{2} {\bs M}_{n-1}\Hh{1},\]
where 
\[{\bs a}\Ll{2}:= \left(\begin{array}{c} \E[I^2]\\0\\0
\end{array}\right) ,\:\: D\Ll{2}:=\left(\begin{array}{ccc}0&0&(\E[G])^2\\0&p\,\E[G]&0\\p^2&0&0\end{array}\right),\:\: \bar D\Ll{2}:=\left(\begin{array}{cc}0&{\mathbb F}[G,I]\\p\,\E[I]&0\\p(1-p)&0\end{array}\right)\hspace{-0.5mm}.\]
We find that
\[{\bs \mu}\Ll{2}:=\E\left[{\bs V}_0\Hh{2}\right]=(I-D\Ll{2})^{-1}\big({\bs a}\Ll{2}+\bar D\Ll{2}{\bs \mu}\Ll{1}\big),\]
which is in line with \eqref{XY} and \eqref{Y2}.
Solving the recursion yields
\[{\bs M}_n\Hh{2} =\sum_{m=0}^{n-1}D\Ll{2}^{\rm m} {\bs f}_{n-1-m} +D\Ll{2}^n {\bs M}_0\Hh{2}, \]
where, relying on \eqref{Mn},
\[{\bs f}_n:={\bs a}\Ll{2}+ \bar D\Ll{2}{\bs \mu}\Ll{1}+ \bar D\Ll{2} D^n\Ll{1}\big({\bs M}_0\Hh{1}-{\bs \mu}\Ll{1}\big).\]
We conclude that, for any $n\in{\mathbb N}$,
\begin{align*}
  {\bs M}_n\Hh{2} =\:&
(I-D\Ll{2})^{-1}(I-D\Ll{2}^n)\big({\bs a}\Ll{2}+ \bar D\Ll{2}{\bs \mu}\Ll{1}\big) + \sum_{m=0}^{n-1}D\Ll{2}^{\rm m}\bar D\Ll{2}D\Ll{1}^{n-1-m}\big({\bs M}_0\Hh{1}-{\bs \mu}\Ll{1}\big)+ D\Ll{2}^n {\bs M}_0\Hh{2}\\
=\:&{\bs \mu}\Ll{2}+\sum_{m=0}^{n-1}D\Ll{2}^{\rm m}\bar D\Ll{2}D\Ll{1}^{n-1-m}\big({\bs V}_0\Hh{1}-{\bs \mu}\Ll{1}\big)+ D\Ll{2}^n\big( {\bs V}_0\Hh{2} - {\bs \mu}\Ll{2}\big).
\end{align*}
Then,
\begin{align*}
{\mathbb S}[{\bs V}_n\Hh{2},{\bs V}_0\Hh{2}]&=\sum_{m=0}^{n-1}D\Ll{2}^{\rm m}\bar D\Ll{2}D\Ll{1}^{n-1-m}\,{\mathbb S}[{\bs V}_0\Hh{1},{\bs V}_0\Hh{2}]+ D\Ll{2}^n\, {\mathbb S}[{\bs V}_0\Hh{2}].
\end{align*}
Observe that, by swapping the order of the sums,
\[\sum_{n=1}^\infty\sum_{m=0}^{n-1}D\Ll{2}^{m}\bar D\Ll{2}D\Ll{1}^{n-1-m}=(I-D\Ll{2})^{-1}\bar D\Ll{2}(I-D\Ll{1})^{-1}=:D\Ll{2}^+\in{\mathbb R}^{3\times 2},\]
where the series converges because $\varrho<1$ implies $\|D\Ll{1}\|<1$ and $\|D\Ll{2}\|<1$.
Denoting ${\bs e}_2:=(1,2,1)^\top$, we thus obtain that
\begin{align*}
    \varsigma\HH{2}{2}=\:&{\bs e}_2^\top \,{{\mathbb S}[{\bs V}_0\Hh{2}]} \,
    {\bs e}_2+2\cdot{\bs e}_2^\top
D\Ll{2}^+\,{\mathbb S}[{\bs V}_0\Hh{1},{\bs V}_0\Hh{2}]\,{\bs e}_2+2\,\cdot{\bs e}_2^\top D\Ll{2}(I-D\Ll{2})^{-1} \,{\mathbb S}[{\bs V}_0\Hh{2}]\,{\bs e}_2\\
=\:&2\cdot{\bs e}_2^\top
D\Ll{2}^+\,{\mathbb S}[{\bs V}_0\Hh{1},{\bs V}_0\Hh{2}]\,{\bs e}_2+{\bs e}_2^\top (I+D\Ll{2})(I-D\Ll{2})^{-1} \,{\mathbb S}[{\bs V}_0\Hh{2}]\,{\bs e}_2<\infty,
\end{align*}
where the finiteness follows from the fact that ${\mathbb S}[{\bs V}_0\Hh{1},{\bs V}_0\Hh{2}]$ and ${\mathbb S}[{\bs V}_0\Hh{2}]$ are entrywise finite if $I$ and $G$ have finite fourth moments. 

(C)~We proceed by computing
\begin{align*}
    \varsigma\HH{2}{1}\equiv  \varsigma\HH{1}{2}&:=\lim_{N\to\infty} N\,{\mathbb C}{\rm ov}(\hat m_{1,N},\hat m_{2,N})= \sum_{n=-\infty}^\infty {\mathbb C}{\rm ov}(Z_0,Z_n^2) \\
    &=\sum_{n=0}^\infty {\mathbb C}{\rm ov}(Z_0^2,Z_n) +\sum_{n=1}^\infty {\mathbb C}{\rm ov}(Z_0,Z_n^2)
\end{align*}
This computation differs from those underlying (A) and (B), in that we can directly exploit our earlier computations. We are to evaluate
\begin{align*}
\Sigma_{n,{\rm I}}\HH{2}{1}&:={\mathbb S}[{\bs V}_n\Hh{1},{\bs V}_0\Hh{2}] ,\quad
\Sigma_{n,{\rm II}}\HH{2}{1}:={\mathbb S}[{\bs V}_n\Hh{2},{\bs V}_0\Hh{1}],
\end{align*}
which are matrices of dimension $2\times 3$ and $3\times 2$, respectively. 
Using the results from part (A), it follows directly that
${\smash \Sigma_{n,{\rm I}}\HH{2}{1}=D\Ll{1}^n \Sigma_{0,{\rm I}}\HH{2}{1}.}$
Analogously, using the results from part (B), 
\[\Sigma_{n,{\rm II}}\HH{2}{1}=\sum_{m=0}^{n-1}D\Ll{2}^{\rm m}\bar D\Ll{2}D\Ll{1}^{n-1-m}\,{\mathbb S}[{\bs V}_0\Hh{1}]+ D\Ll{2}^n \Sigma_{0,{\rm II}}\HH{2}{1}.\]
Upon combining the above,  
\begin{align*}
   \varsigma\HH{2}{1}=&\:{\bs 1}^\top\sum_{n=0}^\infty \Sigma_{n,{\rm I}}\HH{2}{1}\,{\bs e}_2+  {\bs e}_2^\top\sum_{n=1}^\infty \Sigma_{n,{\rm II}}\HH{2}{1}{\bs 1} \\
   =&\: {\bs 1}^\top (I-D\Ll{1})^{-1}\,{\mathbb S}[{\bs V}_0\Hh{1},{\bs V}_0\Hh{2}]\, {\bs e}_2+ {\bs e}_2^\top
D\Ll{2}^+\,{\mathbb S}[{\bs V}_0\Hh{1}]\,{\bs 1}+{\bs e}_2^\top D\Ll{2}(I-D\Ll{2})^{-1} {\mathbb S}[{\bs V}_0\Hh{2},{\bs V}_0\Hh{1}]\,{\bs 1}<\infty;
\end{align*}
here it has been used that $\smash{{\mathbb S}[{\bs V}_0\Hh{1},{\bs V}_0\Hh{2}]}$, $\smash{{\mathbb S}[{\bs V}_0\Hh{1}]}$ and $\smash{{\mathbb S}[{\bs V}_0\Hh{2},{\bs V}_0\Hh{1}]}$ (which is the transpose of $\smash{{\mathbb S}[{\bs V}_0\Hh{1},{\bs V}_0\Hh{2}]}$) are entrywise finite if $I$ and $G$ have finite third moments. 

{(D)~We now determine
\[\varsigma\HH{12}{12}:=\lim_{N\to\infty} N\,{\mathbb V}{\rm ar}[\hat m_{12,N}]={\mathbb V}{\rm ar}[Z_0Z_1] + 2\sum_{n=1}^\infty {\mathbb C}{\rm ov}(Z_0Z_1,Z_nZ_{n+1}). \]
 With with ${\bs V}\Hh{12}_n:=(X_nX_{n+1},X_nY_{n+1}, Y_nX_{n+1},Y_nY_{n+1})^\top$, we now work with the object
\[{\bs M}_{n,0}\Hh{12}:= \E\left[{\bs V}_n\Hh{12}\,|\,{\bs V}_0\Hh{1}\right] \qquad \qquad {\bs M}_{n,1}\Hh{12}:= \E\left[{\bs V}_n\Hh{12}\,|\,{\bs V}_1\Hh{1}\right].\]
By arguments analogous to those used above, we find that
\begin{align*}
    \E\left[{\bs V}\Hh{12}_n\,\Big|\,X_{n}=k,Y_{n}=\ell\right]&= \left(\begin{array}{c}\E[I] \,k + \E[G] \, k \ell \\
    p \, k^2 \\ 
    \E[I] \,\ell + \E[G] \, \ell^2\\
    p \,k\ell\end{array}\right).
\end{align*}
We thus find, with ${\bs M}_{n,1}\Hh{1}:= \E\left[{\bs V}_n\Hh{1}\,|\,{\bs V}_1\Hh{1}\right]$, and ${\bs M}_{n,1}\Hh{2}:= \E\left[{\bs V}_n\Hh{2}\,|\,{\bs V}_1\Hh{1}\right]$
\[{\bs M}_{n,0}\Hh{12} = D\Ll{12} {\bs M}_{n}\Hh{2} + \bar D\Ll{12} {\bs M}_{n}\Hh{1} \qquad\qquad {\bs M}_{n,1}\Hh{12} = D\Ll{12} {\bs M}_{n,1}\Hh{2} + \bar D\Ll{12} {\bs M}_{n,1}\Hh{1},\]
where, 
\begin{align*} D\Ll{12}:=\left(\begin{array}{ccc}0&\E[G]&0\\p&0&0\\0&0&\E[G]\\
0&p&0\end{array}\right),\:\: \bar D\Ll{12}:=\left(\begin{array}{cc}\E[I]&0\\0&0\\0&\E[I]\\
0&0\end{array}\right)\hspace{-0.5mm}.\end{align*}
From this relation, we conclude that 
\[{\bs \mu}\Ll{12} :=\E\left[{\bs V}\Hh{12}_0\right]= \E\left[{\bs V}\Hh{12}_1\right]= D\Ll{12} \, {\bs \mu}\Ll{2}+ \bar D\Ll{12} \, {\bs \mu}\Ll{1}.\]
For any $n\in{\mathbb N}$, following the convention that we define the empty sum to be 0, and suitably modifying and then substituting the expressions for ${\bs M}_{n}\Hh{1}$ and ${\bs M}_{n}\Hh{2}$ found in parts (A) and (B), respectively, we obtain, for $n\in{\mathbb N}$, 
\begin{align*}
   {\bs M}_{n,1}\Hh{12} =\:&{\bs \mu}\Ll{12}+D\Ll{12}  \sum_{m=0}^{n-2}D\Ll{2}^{\rm m}\bar D\Ll{2}D\Ll{1}^{n-2-m}\big({\bs V}_1\Hh{1}-{\bs \mu}\Ll{1}\big) \:+\\&D\Ll{12}D\Ll{2}^{n-1}\big( {\bs V}_1\Hh{2} - {\bs \mu}\Ll{2}\big)
   +\bar D\Ll{12}D\Ll{1}^{n-1}\left({\bs V}_1\Hh{1}-{\bs \mu}\Ll{1}\right)
\end{align*}
The reason we need to condition on ${\bs V}\Hh{1}_1$, rather than on ${\bs V}\Hh{1}_0$, is that
\[
{\mathbb S}[{\bs V}_n\Hh{12},{\bs V}_0\Hh{12}] = \E \left[ {\bs M}_{n,1}\Hh{12} \,  {\bs V}_0\Hh{12}\right] - {\bs \mu}\Ll{12}{\bs \mu}\Ll{12}^\top \neq \E \left[ {\bs M}_{n,0}\Hh{12} \,  {\bs V}_0\Hh{12}\right] - {\bs \mu}\Ll{12}{\bs \mu}\Ll{12}^\top \,.
\]
Hence, conditioning on ${\bs V}\Hh{1}_0$ would not yield the correct covariance expression.
We can now evaluate
\begin{align*}
{\mathbb S}[{\bs V}_n\Hh{12},{\bs V}_0\Hh{12}]=\:&D\Ll{12}\sum_{m=0}^{n-2}D\Ll{2}^{\rm m}\bar D\Ll{2}D\Ll{1}^{n-2-m}\,{\mathbb S}[{\bs V}_1\Hh{1},{\bs V}_0\Hh{12}]\:+\\&D\Ll{12}D\Ll{2}^{n-1}\, {\mathbb S}[{\bs V}_1\Hh{2},{\bs V}_0\Hh{12}] +\bar D\Ll{12}D\Ll{1}^{n-1}\,{\mathbb S}[{\bs V}_1\Hh{1},{\bs V}_0\Hh{12}].
\end{align*}
We conclude that 
\begin{align*}\varsigma\HH{12}{12} =\:& {\bs 1}^\top 
{\mathbb S}[{\bs V}_0\Hh{12}]\,
{\bs 1} + 2\cdot {{\bs 1}^\top D\Ll{12}} D\Ll{2}^+\,
{\mathbb S}[{\bs V}_1\Hh{1},{\bs V}_0\Hh{12}]\,{\bs 1}
\:+\\
&2\cdot {\bs 1}^\top D\Ll{12}(I-D\Ll{2})^{-1}
\, {\mathbb S}[{\bs V}_1\Hh{2},{\bs V}_0\Hh{12}]\,{\bs 1}
+2\cdot {\bs 1}^\top\bar D\Ll{12}(I-D\Ll{1})^{-1}\,
   {\mathbb S}[{\bs V}_1\Hh{1},{\bs V}_0\Hh{12}]\,{\bs 1}<\infty,
\end{align*}
observing that the entries of ${\mathbb S}[{\bs V}_0\Hh{12}]$, 
${\mathbb S}[{\bs V}_1\Hh{2},{\bs V}_0\Hh{12}]$ and $
{\mathbb S}[{\bs V}_1\Hh{1},{\bs V}_0\Hh{12}]$ are finite if $I$ and $G$ have finite fourth moments.


(E)~We next compute
\begin{align*}
    \varsigma\HH{1}{12}&:=\lim_{N\to\infty} N\,{\mathbb C}{\rm ov}(\hat m_{1,N},\hat m_{12,N})= \sum_{n=-\infty}^\infty {\mathbb C}{\rm ov}(Z_0,Z_{n+1}Z_n) \\
    &=\sum_{n=1}^\infty {\mathbb C}{\rm ov}(Z_1Z_0,Z_n) +\sum_{n=0}^\infty {\mathbb C}{\rm ov}(Z_0,Z_{n+1}Z_n) \,.
\end{align*}
Following similar arguments, we obtain
\begin{align*}
    {\bs M}_{n,1}\Hh{1} =\:&{\bs \mu}\Ll{1} + D\Ll{1}^{n-1}\left({\bs V}_1\Hh{1}-{\bs \mu}\Ll{1}\right) \\
    {\bs M}_{n,0}\Hh{12} =\:&{\bs \mu}\Ll{12}+D\Ll{12}  \sum_{m=0}^{n-1}D\Ll{2}^{\rm m}\bar D\Ll{2}D\Ll{1}^{n-1-m}\big({\bs V}_0\Hh{1}-{\bs \mu}\Ll{1}\big) \:+\\&D\Ll{12}D\Ll{2}^{n}\big( {\bs V}_0\Hh{2} - {\bs \mu}\Ll{2}\big)
   +\bar D\Ll{12}D\Ll{1}^{n}\left({\bs V}_0\Hh{1}-{\bs \mu}\Ll{1}\right),
\end{align*}
Using these expressions for ${\bs M}_{n,1}\Hh{1}$ and ${\bs M}_{n,0}\Hh{12}$ together with the results established earlier, we derive
\begin{align*}
{\mathbb S}[{\bs V}_n\Hh{1},{\bs V}_0\Hh{12}]=\:&D\Ll{1}^{n-1}\, {\mathbb S}[{\bs V}_1\Hh{1}, {\bs V}_0\Hh{12}]. \\
{\mathbb S}[{\bs V}_n\Hh{12},{\bs V}_0\Hh{1}]=\:&D\Ll{12}\sum_{m=0}^{n-1}D\Ll{2}^{\rm m}\bar D\Ll{2}D\Ll{1}^{n-1-m}\,{\mathbb S}[{\bs V}_0\Hh{1}]+D\Ll{12}D\Ll{2}^{n}\, {\mathbb S}[{\bs V}_0\Hh{2},{\bs V}_0\Hh{1}] +\bar D\Ll{12}D\Ll{1}^{n}\,{\mathbb S}[{\bs V}_0\Hh{1}]. 
\end{align*}
Combining the above expressions yields
\begin{align*}
   \varsigma\HH{1}{12}=&\:{\bs 1}^\top\sum_{n=1}^\infty {\mathbb S}[{\bs V}_n\Hh{1},{\bs V}_0\Hh{12}]\,{\bs 1} +  {\bs 1}^\top\sum_{n=0}^\infty {\mathbb S}[{\bs V}_n\Hh{12},{\bs V}_0\Hh{1}]\,{\bs 1} \\
   =&\: {\bs 1}^\top (I-D\Ll{1})^{-1}\,{\mathbb S}[{\bs V}_1\Hh{1},{\bs V}_0\Hh{12}]\, {\bs 1}+ {\bs 1}^\top
D\Ll{12} D\Ll{2}^+\,{\mathbb S}[{\bs V}_0\Hh{1}]\,{\bs 1}\:+\\&\: {\bs 1}^\top D\Ll{12}(I-D\Ll{2})^{-1}
\, {\mathbb S}[{\bs V}_0\Hh{2},{\bs V}_0\Hh{1}]\,{\bs 1}+{\bs 1}^\top\bar D\Ll{12}(I-D\Ll{1})^{-1}\,
   {\mathbb S}[{\bs V}_0\Hh{1}]\,{\bs 1}<\infty,
\end{align*}
where finiteness follows from the fact that all entries of ${\mathbb S}[{\bs V}_0\Hh{1}]$, ${\mathbb S}[{\bs V}_1\Hh{1},{\bs V}_0\Hh{12}]$, ${\mathbb S}[{\bs V}_0\Hh{2},{\bs V}_0\Hh{1}]$ are finite whenever the random variables $I$ and $G$ have finite third moments.

(F)~We finally proceed by an analogous line of argument to evaluate
\begin{align*}
    \varsigma\HH{2}{12}&:=\lim_{N\to\infty} N\,{\mathbb C}{\rm ov}(\hat m_{2,N},\hat m_{12,N})= \sum_{n=-\infty}^\infty {\mathbb C}{\rm ov}(Z_0^2,Z_{n+1}Z_n) \\
    &=\sum_{n=1}^\infty {\mathbb C}{\rm ov}(Z_1Z_0,Z_n^2) +\sum_{n=0}^\infty {\mathbb C}{\rm ov}(Z_0^2,Z_{n+1}Z_n)
\end{align*}
Using
\[
  {\bs M}_{n,1}\Hh{2} =\: {\bs \mu}\Ll{2}+\sum_{m=0}^{n-2}D\Ll{2}^{\rm m}\bar D\Ll{2}D\Ll{1}^{n-2-m}\big({\bs V}_1\Hh{1}-{\bs \mu}\Ll{1}\big)+ D\Ll{2}^{n-1}\big( {\bs V}_1\Hh{2} - {\bs \mu}\Ll{2}\big)
\]
together with the representations of ${\bs M}_{n,0}\Hh{12}$, one can derive the following covariance relations
\begin{align*}
{\mathbb S}[{\bs V}_n\Hh{2},{\bs V}_0\Hh{12}]=\:&\sum_{m=0}^{n-2}D\Ll{2}^{\rm m}\bar D\Ll{2}D\Ll{1}^{n-2-m}\,{\mathbb S}[{\bs V}_1\Hh{1},{\bs V}_0\Hh{12}] + D\Ll{2}^{n-1}\, {\mathbb S}[{\bs V}_1\Hh{2}, {\bs V}_0\Hh{12}], \\
{\mathbb S}[{\bs V}_n\Hh{12},{\bs V}_0\Hh{2}]=\:&D\Ll{12}\sum_{m=0}^{n-1}D\Ll{2}^{\rm m}\bar D\Ll{2}D\Ll{1}^{n-1-m}\,{\mathbb S}[{\bs V}_0\Hh{1}, {\bs V}_0\Hh{2}]\:+\\&D\Ll{12}D\Ll{2}^{n}\, {\mathbb S}[{\bs V}_0\Hh{2}] +\bar D\Ll{12}D\Ll{1}^{n}\,{\mathbb S}[{\bs V}_0\Hh{1}, {\bs V}_0\Hh{2}]. 
\end{align*}
Putting these components together, we arrive at
\begin{align*}
   \varsigma\HH{2}{12}=&\:{\bs e}_2^\top\sum_{n=1}^\infty {\mathbb S}[{\bs V}_n\Hh{2},{\bs V}_0\Hh{12}]{\bs 1} +  {\bs 1}^\top\sum_{n=0}^\infty {\mathbb S}[{\bs V}_n\Hh{12},{\bs V}_0\Hh{2}]{\bs e}_2 \\
   =&\: {\bs e}_2^\top
D\Ll{2}^+\,{\mathbb S}[{\bs V}_1\Hh{1}, {\bs V}_0\Hh{12}]{\bs 1} + {\bs e}_2^\top (I-D\Ll{2})^{-1} \,{\mathbb S}[{\bs V}_0\Hh{2}, {\bs V}_0\Hh{12}]\,{\bs 1} \:+\\&\:
{\bs 1}^\top D\Ll{12}
D\Ll{2}^+\,{\mathbb S}[{\bs V}_0\Hh{1}, {\bs V}_0\Hh{2}]\,{\bs e}_2 + {\bs 1}^\top D\Ll{12}(I-D\Ll{2})^{-1}
\, {\mathbb S}[{\bs V}_0\Hh{2}]\,{\bs e}_2
\:+\\& \:{\bs 1}^\top\bar D\Ll{12}(I-D\Ll{1})^{-1}\,
   {\mathbb S}[{\bs V}_0\Hh{1}, {\bs V}_0\Hh{2}]\,{\bs e}_2<\infty;
\end{align*}
the finiteness is due to the fact that all entries of the matrices $\smash{{\mathbb S}[{\bs V}_1\Hh{1}, {\bs V}_0\Hh{12}]}$, $\smash{{\mathbb S}[{\bs V}_0\Hh{2}, {\bs V}_0\Hh{12}]}$, $\smash{{\mathbb S}[{\bs V}_0\Hh{1}, {\bs V}_0\Hh{2}]}$ and $\smash{{\mathbb S}[{\bs V}_0\Hh{2}]}$
are finite whenever the random variables $I$ and $G$ have finite fourth moments.
}

\subsection{Algorithm for determining stationary moments}
The entries of the covariance matrix $\Sigma_\varsigma$, as identified in Subsection \ref{covm}, require access to the objects
\begin{equation}\label{eq:mom}
\smash{{\mathbb S}[{\bs V}_0\Hh{1}]},\quad \smash{{\mathbb S}[{\bs V}_0\Hh{2}]},\quad\smash{{\mathbb S}[{\bs V}_0\Hh{12}]},\quad \smash{{\mathbb S}[{\bs V}_0\Hh{1}, {\bs V}_0\Hh{2}]},\quad\smash{{\mathbb S}[{\bs V}_0\Hh{1}, {\bs V}_0\Hh{12}]},\quad\smash{{\mathbb S}[{\bs V}_0\Hh{2}, {\bs V}_0\Hh{12}]}.\end{equation}
The entries of these matrices can be evaluated once the moments
\[\varphi_{k\ell}:=\E\left[X^kY^\ell\right],\]
are available, where $(X,Y)$ follows the limiting distribution of $(X_n,Y_n)_{n\ge 0}$ as $n\to\infty$, under the stability condition $\varrho<1$.
For small values of $k$ and $\ell$, explicit expressions for $\varphi_{k\ell}$ can be derived in a straightforward manner. However, as $k$ and $\ell$ increase, the resulting expressions quickly become cumbersome. In particular, our analysis requires moments satisfying $k+\ell=4$, which leads to lengthy and unwieldy formulas. To address this difficulty, we therefore propose a recursive algorithm that allows arbitrary moments $\varphi_{k\ell}$ to be computed efficiently.

The starting point of the algorithm is the obvious identity, that is a direct consequence of the tower property,
\begin{equation}\label{eq:condexp}
\E[X_{n+1}^k Y_{n+1}^\ell] = \E\left[\E[X_{n+1}^k Y_{n+1}^\ell\,|\,X_n,Y_n]\right] = \E\left[\E[X_{n+1}^k\,|\,Y_n] \,\E[Y_{n+1}^\ell\,|\,X_n]\right];
\end{equation}
here it is observed that, conditional on $(X_n,Y_n)$, (a)~the random variables $X_{n+1}$ and $Y_{n+1}$ are independent, and (b)~$X_{n+1}$ depends only on $Y_n$, while $Y_{n+1}$ depends only on $X_n$.

We proceed by evaluating $\E[X_{n+1}^k\,|\,Y_n]$ and $\E[Y_{n+1}^\ell\,|\,X_n]$.
As a direct application of the binomial theorem, we obtain
\begin{align} \E[X_{n+1}^k\,|\,Y_n] &= \displaystyle \E\left[ \sum_{m=0}^k \binom{k}{m} I^{k-m} \left( \sum_{i=1}^{Y_n} G_{i} \right)^m \, \Bigg| \, Y_n \right]\notag \\ &= \displaystyle \sum_{m=0}^k \binom{k}{m} \E [ I^{k-m} ] \,\E\left[ \left( \sum_{i=1}^{Y_n} G_{i} \right)^m \,\Bigg|\, Y_n \right]\notag  \\ &= \E[I^k]+\displaystyle \sum_{m=1}^k \binom{k}{m} \E[ I^{k-m} ] \left( \sum_{j=1}^m (Y_n)_j \,B_{m,j} \big( \E[G],\E[G^2],...,\E[G^{m-j+1}]\big)\right) , \label{eq:exp1} \end{align}
where $(Y_n)_j := Y_n(Y_n-1)\cdots(Y_n-j+1)$ for $j \ge 1$, with $(Y_n)_0 := 1$, denotes the {\it falling factorial}, and $B_{m,j}:{\mathbb R}^{m-j+1}\to{\mathbb R}$ are the {\it partial Bell polynomials} \cite{PeccatiTaqqu2011}. Specifically, with $g_j := \E[G^j]$ and $S := \sum_{i=1}^y G_i$ for some $y\in{\mathbb N}$, the cases relevant for us are, besides the obvious $\E[S] = y\,g_1,$
\begin{align*}
 \E[S^2] =\:& y\,g_2 + y(y-1)\,g_1^2, \\[1ex]
 \E[S^3] =\:& y\,g_3 + 3\,y(y-1)\,g_2 g_1 + y(y-1)(y-2)\,g_1^3, \\[1ex]
 \E[S^4] =\:& y\,g_4 
+ 4\,y(y-1)\,g_3 g_1 
+ 3\,y(y-1)\,g_2^2
\:+\\
&6\,y(y-1)(y-2)\,g_2 g_1^2 
+ y(y-1)(y-2)(y-3)\,g_1^4.
\end{align*}
In general, the above reasoning reveals that, for suitably chosen coefficients $\beta^\circ_{jm}$,
\[\E\left[ \left( \sum_{i=1}^{Y_n} G_{i} \right)^m \,\Bigg|\, Y_n \right] = \sum_{j=1}^m \beta^\circ_{jm}\,Y_n^j.\]
This implies that, upon interchanging the order of summation,  we can rewrite \eqref{eq:exp1} as, with $\beta_{0k}:=\E[I^k]$ and $\smash{\beta_{jk}:=\sum_{m=j}^k\binom{k}{m}\,\E[I^{k-m}]\,\beta^\circ_{jm}}$ for $j\in\{1,\ldots,m\}$, 
\begin{equation}\label{EX}\E[X_{n+1}^k\,|\,Y_n]=\sum_{j=0}^k\beta_{jk}\,Y_n^j.
\end{equation}

We proceed by analyzing $\E[Y_{n+1}^\ell\,|\,X_n]$. Recall that, conditional on $X_n$, we have that $Y_{n+1}$ is binomially distributed with parameters $X_n$ and $p$. Hence,  with $S_{\ell,j}$ denoting the {\it Stirling numbers of the second kind},
\[ 
\E[Y_{n+1}^\ell\,|\,X_n] = \sum_{j=0}^\ell S_{\ell,j}\, (X_n)_j \, p^j.
\]
Specifically, if $B\sim {\rm Bin}(x,p)$ for some $x\in{\mathbb N}$, the cases relevant for us are, besides the obvious $\E[B]=x\,p$,
\begin{align*}
\E[B^2] &= x\,p + x(x-1) p^2,\\[1ex]
\E[B^3] &= x\,p + 3\,x(x-1) p^2 + x(x-1)(x-2) p^3,\\[1ex]
\E[B^4] &= x\,p + 7\,x(x-1) p^2 + 6\,x(x-1)(x-2) p^3 + x(x-1)(x-2)(x-3) p^4.
\end{align*}
 In general, the above reasoning reveals that we can identify coefficients $\gamma_{i\ell}$ such that 
 \begin{equation}\label{EY}\E[Y_{n+1}^\ell\,|\,X_n]=\sum_{i=0}^\ell\gamma_{i\ell}\,X_n^i.
 \end{equation}
 Upon combining \eqref{eq:condexp} with \eqref{EX} and \eqref{EY}, the following lemma follows.
 \begin{lemma} For any $k,\ell\ge 1$,
     \[\varphi_{k\ell}=\sum_{j=0}^k\sum_{i=0}^\ell \beta_{jk}\gamma_{i\ell}\,\varphi_{ij}. \]
 \end{lemma}

 Now define
 \[\omega_{k\ell}:= \left(\sum_{j=0}^k\sum_{i=0}^\ell \beta_{jk}\gamma_{i\ell}\,\varphi_{ij}\right) - 
 \beta_{kk}\gamma_{\ell\ell}\,\varphi_{\ell k}.\]
 First fix a $k$ and $\ell$ such that $k\not =\ell$. Suppose that we have computed $\varphi_{ij}$ for all $(i,j)\not=(k,\ell)$ such that $i\le k$ and $j\le \ell$, and for all $(i,j)\not=(\ell,k)$ such that $i\le \ell$ and $j\le k$. This means that we know both $\omega_{k\ell}$ and $\omega_{\ell k}.$ Then
 \[\varphi_{k\ell}=\omega_{k\ell}+ \beta_{kk}\gamma_{\ell\ell}\,\varphi_{\ell k},\quad \varphi_{\ell k}=\omega_{\ell k}+ \beta_{\ell\ell}\gamma_{kk}\,\varphi_{k\ell}, \]
 so that
 \[\left(\begin{array}{c}
   \varphi_{k\ell}\\\varphi_{\ell k}
 \end{array}\right)=\left(\begin{array}{cc}
    1  & -\beta_{kk}\gamma_{\ell\ell} \\
     -\beta_{\ell\ell}\gamma_{kk} & 1
 \end{array}\right)^{-1}\left(\begin{array}{c} \omega_{k\ell}\\\omega_{\ell k}
 \end{array}\right).
 \]
 The case $k=\ell$ is simpler. Suppose that we have computed $\varphi_{ij}$ for all $(i,j)\not=(k,k)$ such that $i\le k$ and $j\le k$. It leads to the equation $\varphi_{kk}=\omega_{kk}+ \beta_{kk}\gamma_{kk}\,\varphi_{k k}$, so that
 \[\varphi_{kk} = \frac{\omega_{kk}}{1-\beta_{kk}\gamma_{kk}}.\]
 We can now calculate all moments that we need in \eqref{eq:mom}. For instance, an order that works is, in self-evident notation, $00\rightsquigarrow (01,10)\rightsquigarrow 11\rightsquigarrow (02,20)\rightsquigarrow(12,21)\rightsquigarrow 22\rightsquigarrow (03,30)\rightsquigarrow(13,31)\rightsquigarrow(04,40).$

\subsection{Numerical study}

We conduct a numerical study to assess the performance of the proposed method-of-moments estimators, under which the unknown parameters are $\lambda$, $\nu$, and $p$. A total of $1{,}000$ independent simulation runs are performed, and in each run a sample path of length $N = 10^5$ observation time points is generated, from which the corresponding method-of-moments estimates $(\hat{\lambda}_N, \hat{\nu}_N, \hat{p}_N)$ are computed.

The collection of estimates over the $1{,}000$ runs yields an empirical distribution for each parameter. To assess the distributional behavior of the estimators, Q--Q plots of the standardized estimation errors are constructed, where each estimator is centered at its true value and scaled by its empirical standard deviation across runs. The resulting standardized samples are compared against the standard normal distribution. Figure \ref{fig:qq0} confirms that the proposed estimators are asymptotically normal.
\begin{figure}[h]
\centering

\subcaptionbox{$\lambda=0.5$}
{\includegraphics[width=0.32\linewidth]{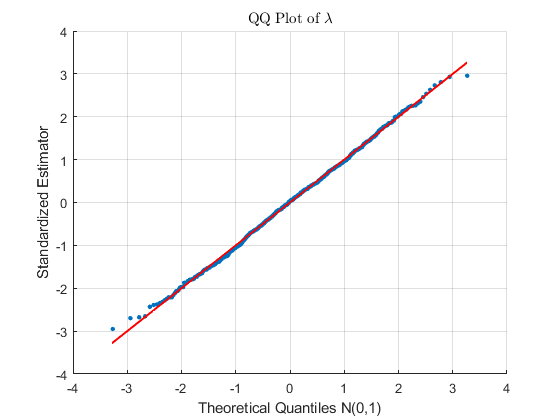}}
\subcaptionbox{$\nu=2$}
{\includegraphics[width=0.32\linewidth]{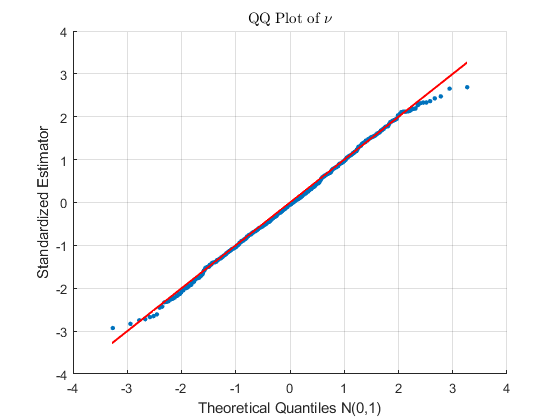}}
\subcaptionbox{$p = 0.3$}
{\includegraphics[width=0.32\linewidth]{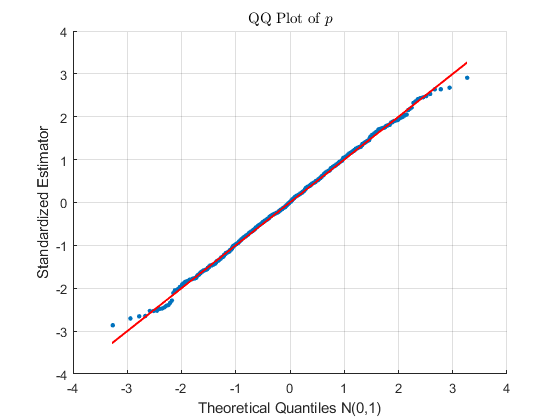}}

\caption{Normal Q–Q plots of the standardized estimators for $\lambda$, $\mu$, and $p$, where each estimator is centered at its true value (shown below each plot) and scaled by its empirical standard deviation computed from 1,000 independent replications.}
\label{fig:qq0}
\end{figure}

\section{Extensions}\label{sec:ext}

In this section, we discuss two types of extensions: (i) one concerning alternative observation schemes, and (ii) one allowing for more than two age groups.

\subsection{Different observation schemes} In this subsection, we consider two observation schemes: one in which only the adults are observed, and one in which only the juveniles are observed. Remarkably, in the `Poisson setting' \eqref{P}, observing the juvenile time series $(X_n)_{n\ge 0}$ allows us to identify all three model parameters, whereas observing the adult time series $(Y_n)_{n\ge 0}$ leads to identifiability issues that prevent full parameter estimation.

\subsubsection{Observing the adults} We first consider the situation in which the observed process is $\{Y_n\}_{n\geq 0}$. We have that $\E[Y_{n+1}Y_n] = (\E[Y_n])^2$, as a consequence of the easily verifiable relation
 \[
    \E[Y_{n+1}Y_n] = \E\left[\E\big[{\rm Bin}(X_n,p)Y_n\,\vert \,X_n\right]\big] = p\, \E[X_nY_n]
    \]
    and \eqref{momPois}.
    As a consequence, we cannot  base our estimator on  $(\E[Y_n], \E[Y_n^2], \E[Y_{n+1}Y_n])$, since $\E[Y_n]$ and $\E[Y_{n+1}Y_n]$ provide essentially the same information.
The following lemma shows that, in fact, the same identifiability issue arises if one replaces $\E[Y_{n+1}Y_n]$ by $\E[Y_{n+k}Y_n]$ for any $k \geq 1$. Define
$\GI := p \LI$ and recall that $\GG = p \LG$.

\begin{lemma}\label{L3}
Assume that \eqref{P} and \eqref{S} hold. Then the first and second moments of the observed adult population, namely $\E[Y_n]$ and $\E[Y_n^2]$, as well as the cross-moments $\E[Y_{n+k}Y_n]$ for any $k \geq 1$, are functions of the parameters $\GI$ and $\GG$ only.
\end{lemma}
\begin{proof}
    By \eqref{momPois}--\eqref{momPois2}, we obtain that $\E[Y_n]$ and $\E[Y_n^2]$ are functions of  $\gamma$ and $\varrho$:
    \[
    \E[Y_n] = \dfrac{\GI}{1-\GG}, \qquad \E[Y_n^2] = \dfrac{\GI}{1-\GG} + \dfrac{\GI^2}{1-\GG^2} + \dfrac{\GI\GG(\GG+2\GI)}{(1-\GG)(1-\GG^2)}.
    \]
    It remains to show that $\E[Z_{n+k}Z_n]$ depends solely on $\GI$ and $\GG$. We do this by by induction over $k\geq1$. For $k=1$, we already saw that $\E[Y_{n+1}Y_n] = (\E[Y_n])^2$.
    Assume now that the claim holds for $k\geq1$ and we prove it for $k+1$. We find
   \begin{align*}
\E[Y_{n+k+1}Y_n] 
&= \E\!\left[\E\!\left[{\rm Bin}(X_{n+k},p)\,Y_n \,\big\vert\, X_{n+k}\right]\right] = p\,\E[X_{n+k}Y_n] \\
&= p\,\E\!\left[\E\!\left[\left( I_{n+k-1} 
      + \sum_{i=1}^{Y_{n+k-1}} G_{n+k-1,i} \right) Y_n 
      \,\Bigg\vert\, Y_{n+k-1} \right]\right] = \GI \,\E[Y_n] + \GG \,\E[Y_{n+k-1}Y_n].
\end{align*}
If $k\geq2$, then the claim follows by the induction hypothesis, otherwise it follows from the explicit expression of $\E[Y_n^2]$. This concludes the proof.
\end{proof}

A direct consequence of Lemma \ref{L3} is that, when using moments of the form $\E[Y_n]$, $\E[Y_n^2]$, and $\E[Y_{n+k}Y_n]$, one can estimate only the combined quantities $p\lambda$ and $p\nu$, but not the three parameters $p$, $\lambda$, and $\nu$ individually. Lemma \ref{L4} below further shows that considering higher-order moments $\E[Y_n^k]$ alone does not resolve this issue, since the stationary distribution of $Y_n$ also depends solely on $\GI$ and $\GG$. Define by $\Psi_Y(\cdot)$ the probability generating function of $Y_n$ in stationarity. Let $\phi_\varrho^{(j)}(\cdot)$ denote the $j$-th iterate of the map $\phi_\varrho(\cdot)$ defined by $x \mapsto \exp(\GG(x-1))$.

\begin{lemma}\label{L4}
Assume that \eqref{P} and \eqref{S} hold. Then,
\[
\Psi_Y(s) = \exp\Biggl( \GI \sum_{j=0}^{\infty} \bigl( \phi_\varrho^{(j)}(s) - 1 \bigr) \Biggr).
\]
In particular, the stationary distribution of $Y_n$ depends solely on $\GI$ and $\GG$.
\end{lemma}

\begin{proof}
    By the Poisson thinning property, we know that if $V\sim{\rm Poisson}(\lambda)$ and $W\,|\,V\sim{\rm Bin}(V,p)$, then $W\sim{\rm Poisson}(p\lambda)$. This implies that $Y_{n+1}\sim{\rm Poisson}(\GI+\GG Y_n)$. Let us denote by $\Psi_{Y_n}(\cdot)$ the probability generating function of $Y_n$, i.e., $\Psi_{Y_n}(s)=\E[s^{Y_n}]$, $s\in[0,1]$. Thus,
    \[
    \begin{array}{ll}
    \Psi_{Y_{n+1}}(s) &= \E\left[\E\left[ s^{Y_{n+1}} \vert\, Y_n \right]\right] = \E\left[ \exp(\GI+\GG Y_n) (s-1) \right] \\
    &= \exp(\GI(s-1)) \Psi_{Y_n}(\exp(\GG(s-1))).
    \end{array}
    \]
    In stationarity, we thus find
    \begin{equation}\label{eq:Gstationary}
    \Psi_Y(s) = {\rm e}^{\GI(s-1)} \Psi_Y\left( {\rm e}^{\GG(s-1)} \right).
    \end{equation}
   The goal is now to express $\Psi_Y(s)$ in a form that depends only on $\GI$ and $\GG$, thereby showing that the stationary distribution of the process $\{Y_n\}_{n \ge 0}$ also depends only on these quantities. To this end, consider the function
$\phi_\varrho(s) := \exp\bigl(\GG(s-1)\bigr)$, 
which maps the interval $[0,1]$ into itself. It is straightforward to verify that $s^* = 1$ is a fixed point of the equation $\phi_\varrho(s) = s$, i.e., $\phi_\varrho(s^*) = s^*$.
By the Banach fixed point theorem, this fixed point is unique. Indeed, $\phi_\varrho$ is a contraction, as $\phi_\varrho'(s) = \GG \exp(\GG(s-1)) \le \GG < 1$; here we have used that we imposed the stationary condition $\varrho<1$. Moreover, the theorem guarantees that the sequence defined recursively by $s_{k+1} = \phi_\varrho(s_k)$, for $k \ge 0$ and any initial $s_0 \in [0,1]$, converges to the fixed point:
\[
\lim_{k \to \infty} s_k = \lim_{k \to \infty} \phi_\varrho^{(k)}(s_0) = 1,
\]
where $\phi_\varrho^{(k)}$ denotes the $k$-th iterate of $\phi_\varrho$. By iterating \eqref{eq:Gstationary} $k$ times, we find
    \[
    \Psi_Y(s) = \exp\left( \GI \displaystyle\sum_{j=0}^{k-1}\left( \phi_\varrho^{(j)}(s)-1 \right) \right) \Psi_Y\left( \phi_\varrho^{(k)}(s) \right).
    \]
    Since $\Psi_Y(s)$ is continuous at $s=1$ and $\Psi_Y(1)=1$, we conclude
    \[
    \begin{array}{ll}
     \Psi_Y(s) &= \displaystyle\lim_{k\to\infty} \exp\left( \GI \displaystyle\sum_{j=0}^{k-1}\left( \phi_\varrho^{(j)}(s)-1 \right) \right) \Psi_Y\left( \phi_\varrho^{(k)}(s) \right) = \exp\left( \GI \displaystyle\sum_{j=0}^{\infty}\left( \phi_\varrho^{(j)}(s)-1 \right) \right).
    \end{array}
    \]
    Observing that the function $\phi_\varrho(\cdot)$ depends only on $\GG$, this concludes the proof.
\end{proof}

\subsubsection{Observing the juveniles} We proceed by considering the case that we observe $(X_n)_{n\ge 0}$. 
Similar to what we found for the case of observing $(Y_n)_{n\ge 0}$, we find that
\[\E[X_{n+1}X_n] = \LI \E[X_n] + \LG \E[X_nY_n] = \frac{\LI^2 }{ 1-p\LG} + \frac{p\LG\LI^2 }{ (1-p\LG)^2} = \frac{\LI^2 }{ (1-p\LG)^2} =(\E[X_n])^2,\]
so that the moment equations based on $\E[X_n]$ and $\E[X_{n+1}X_n]$ effectively provide the same information. Unlike in the case of observing $(Y_n)_{n\ge 0}$, this complication can be remedied by working with the moments $(\E[X_n],\E[X_n^2], \E[X_{n+2}X_n])$, where we use the (readily verified) relation
\[\E[X_{n+2}X_n]= \LI  \E[X_n] + p  \LG + \E[X_n^2] .\]
In particular, recalling the relations \eqref{momPois}--\eqref{momPois2}, it follows that $\E[X_{n+2}X_n]$ cannot be expressed in terms of $\gamma$ and $\varrho$ only.  The next lemma shows that, more generally, the stationary distribution of $X_n$ does not depend only on $\GI$ and $\GG$. Define by $\Psi_X(\cdot)$ the probability generating function of $X_n$ in stationarity. Recall that $\phi_\varrho^{(j)}(\cdot)$ denotes the $j$-th iterate of the map $\phi_\varrho(\cdot)$ defined by $x \mapsto \exp(\GG(x-1))$.

\begin{lemma}
    Assume that \eqref{P} and \eqref{S} hold. Then,
    \[
    \Psi_X (s) = \exp\left( \LI(s-1) \right) \exp\left( \GI \displaystyle\sum_{j=0}^{\infty}\left( \phi_\varrho^{(j)}\left( \exp\left(\LG(s-1) \right) \right)-1 \right) \right).
    \]
  In particular, the stationary distribution of $Y_n$ does not depend solely on $\GI$ and $\GG$.
\end{lemma}
\begin{proof}
    By noting that $X_{n+1}\sim{\rm Poisson}(\LI+\LG Y_n)$, the probability generating function $\Psi_{X_{n+1}}(\cdot)$ of $X_{n+1}$ can be written as
    \[
    \begin{array}{ll}
    \Psi_{X_{n+1}}(s) &= \E\left[ \E \left[ s^{X_{n+1}} \, \vert\, Y_n \right]\right] = \exp(\LI(s-1)) \E[\exp(\LG(s-1)Y_n)] \\
    &= \exp(\LI(s-1)) \Psi_{Y_{n}}(\exp(\LG(s-1))).
    \end{array}
    \]
    As a consequence, in stationarity,  
    \[
    \begin{array}{ll}
    \Psi_{X}(s) &= \exp(\LI(s-1)) \Psi_{Y}(\exp(\LG(s-1))) \\
    &= \exp\left( \LI(s-1) \right) \exp\left( \GI \displaystyle\sum_{j=0}^{\infty}\left( \phi_\varrho^{(j)}\left( \exp\left(\LG(s-1) \right) \right)-1 \right) \right).
    \end{array}
    \]
    Since the prefactor depends explicitly on $\LI$, and the argument of the iterates $\phi_\varrho^{(j)}$, with $j\geq0$, depends explicitly on $\LG$, the stationary distribution of $X_n$ does not depend only on $\GI$ and $\GG$.  
\end{proof}

\begin{remark}{\em 
The asymmetry in the limiting distributions of $Y_n$ and $X_n$, as observed above, stems from the fundamentally different roles played by thinning and Poisson summation in their recursive definitions. At each step, $Y_n$ is thinned by $p$, and therefore the thinning factor $p$ appears linearly in the stationary probability generating function. This makes the stationary distribution depend only on $\GI$ and $\GG$. On the other hand, $X_n$ is the sum of $Y_{n-1}+1$ Poisson random variables, so the  parameters $\LI$ and $\LG$ enter directly.  This means that the stationary distribution of $X_n$ cannot be expressed in terms of $\GI$ and $\GG$ alone.    }\hfill $\spadesuit$
\end{remark}

\subsection{More age groups}
In the model considered so far, we distinguished between juveniles and a single adult group. The adult population was assumed to be homogeneous with respect to fertility, and immigration was allowed only into the juvenile class. In this section, we show that this framework can be generalized in several natural and important directions.

First, instead of working with only two age classes, we may consider $K$ age groups for any $K \in \mathbb{N}$, and allow immigration to occur in any of these groups. Second, each age group may contribute to the production of juveniles through its own offspring distribution.

The purpose of this section is to demonstrate that the analysis developed in Sections \ref{sec:mod}--\ref{sec:norm} extends to this considerably more general setting. We emphasize, however, that this extension comes at the cost of substantially more involved computations, which we have chosen not to present in detail, as they offer limited additional insight. For example, in the discussion of asymptotic normality in Section \ref{sec:norm}, the associated covariance matrix contains $\tfrac{1}{2}(K+1)(K+2)$ distinct elements, each of which must be computed explicitly. It is worth noting, however, that from a conceptual point of view, these calculations closely mirror those developed in Section $\ref{sec:norm}$.

In this general setup, the branching process under study has $K$ types: the juveniles that are recorded via the process ${\smash \{X_n\}_{n\ge 0}}$, and the $k$-group of adults by ${\smash \{Y^{(k)}_n\}_{n\ge 0}}$, for $k=1,\ldots,K$. Then the dynamics are governed by the following $(K+1)$-dimensional stochastic recursion:
\[
\begin{aligned}
X_{n+1} &= I_n^{(0)} + \sum_{k=1}^K \sum_{i=1}^{Y_n^{(k)}} G^{(k)}_{n,i}, 
&\hspace{1.3cm}
\begin{aligned}
Y_{n+1}^{(1)} &\sim I_n^{(1)} + \mathrm{Bin}(X_n, p_1),\\
Y_{n+1}^{(2)} &\sim I_n^{(2)} + \mathrm{Bin}(Y_n^{(1)}, p_2),\\
&\ \vdots\\
Y_{n+1}^{(K)} &\sim I_n^{(K)} + \mathrm{Bin}(Y_n^{(K-1)}, p_K),
\end{aligned}
\end{aligned}
\]
where the random sequences involved in the model are such that
\[
\begin{aligned}
I_n^{(k)} &\sim \text{i.i.d.\ \,as } I^{(k)}\in{\mathbb N}_0, &\quad & n \ge 0, \quad &k = 0,\dots,K,\\
G_{n,i}^{(k)} &\sim \text{i.i.d.\ \,as } G^{(k)}\in{\mathbb N}_0, &\quad & n,i \ge 1, \quad &k = 1,\dots,K,
\end{aligned}
\]
with all sequences $\{I_n^{(k)}\}_{n \ge 0}$ and $\{G_{n,i}^{(k)}\}_{n \ge 0}$ being mutually independent.

We now explain how to compute the stationary mean of the process, using a calculation that simultaneously reveals the model’s stability condition. It is readily checked that in stationarity one has, with ${\bs I}\equiv (I^{(0)},\ldots, I^{(K)})^\top$ and ${\bs W}_n\equiv (X_n,Y_n^{(1)},\ldots, Y_n^{(K)})$,
\[\E[{\bs W}_{n+1}]= \E[{\bs I}] + D^{(K)} \,\E[{\bs W}_n],\]
where
\begin{equation*}
    D^{(K)}:= \left(\begin{array}{cccccc}
    0&\E[G^{(1)}]&\E[G^{(2)}]&\cdots&\E[G^{(K-1)}]&\E[G^{(K)}]\\
    p_1&0&0&\cdots&0&0\\
    0&p_2&0&\cdots&0&0\\
    \vdots&\vdots&\vdots&\ddots&\vdots&\vdots\\
    0&0&0&\cdots&p_K&0\end{array}\right).
\end{equation*}
It follows that the stationary mean is, provided that spectral radius of $D^{(K)}$ is strictly smaller than $1$,
\[\E[{\bs W}_n]= \left(I-D^{(K)}\right)^{-1}\E[{\bs I}].\]
{Appendix~\ref{ThreeAge} presents a worked example for \(K=2\), including detailed computations of the underlying moment expressions.}

\section{Discussion and concluding remarks}\label{sec:disc}
In this paper, we developed a method-of-moments approach for estimating the parameters of an age-dependent branching process from observations of the total population size alone. In the ergodic regime, we established the asymptotic normality of the resulting estimator and illustrated that the approach extends naturally to several more general settings, including models with multiple age groups. We hope that these results provide a useful basis for statistical inference in age-structured branching models and may serve as a starting point for further developments in more general observation schemes or model classes.

Several directions for future research appear promising. A natural extension is the statistical analysis of sex-structured age-dependent branching processes, where the interaction between age and sex raises new modeling and identifiability challenges. Another interesting direction is to study estimation under more general observation schemes, such as irregular or noisy observations, or in the presence of missing data. It would also be worthwhile to investigate inference beyond the ergodic regime, where the long-term behavior of the process is fundamentally different, and to develop likelihood-based or Bayesian estimation procedures that can be compared, both theoretically and numerically, with the method-of-moments approach considered here.

\appendix

\section{Explicit expressions for stationary moments}

We now illustrate how to compute $\smash{{\mathbb S}[{\bs V}_0\Hh{1}]}$, $\smash{{\mathbb S}[{\bs V}_0\Hh{2}]}$, $\smash{{\mathbb S}[{\bs V}_0\Hh{12}]}$, $\smash{{\mathbb S}[{\bs V}_0\Hh{1}, {\bs V}_0\Hh{2}]}$, $\smash{{\mathbb S}[{\bs V}_0\Hh{1}, {\bs V}_0\Hh{12}]}$, $\smash{{\mathbb S}[{\bs V}_0\Hh{2}, {\bs V}_0\Hh{12}]}$. To derive these quantities, we first define the following auxiliary variables:
\begin{align*}
    & {\bs N}_n\Hh{1}:=(X_n, Y_n)^\top \\
    & {\bs N}_n\Hh{2}:= (X_n^2 , X_nY_n, Y_n^2)^\top \\
    & {\bs N}_n\Hh{3}:= (X_n^3 , X_n^2Y_n, X_nY_n^2, Y_n^3)^\top \\
    & {\bs N}_n\Hh{4}:= (X_n^4 , X_n^3Y_n, X_n^2Y_n^2,  X_nY_n^3, Y_n^4)^\top \,.
\end{align*}
Equation \eqref{X and Y} provides the expression for $\E{\bs N}_n\Hh{1}$. We now show how $\E{\bs N}_n\Hh{2}$ can be computed based on $\E{\bs N}_n\Hh{1}$.

(I) We observe that ${\bs N}_n\Hh{2} = {\bs V}\Hh{2}_n$. As established earlier in part (B), we have
\begin{align*}
    \E\left[{\bs V}\Hh{2}_n\,\Big|\,X_{n-1}=k,Y_{n-1}=\ell\right]&=\left(\begin{array}{c}\E[I^2]+{\mathbb F}[G,I]\,\ell+(\E[G])^2\ell^2\\p\,\E[I]\,k+p\,\E[G]\,k\ell\\p(1-p)\,k+p^2 \,k^2\end{array}\right) \\
    &= {\bs b}_2 + C\LL{2}{1} \left(\begin{array}{c} k \\ \ell\end{array}\right) + C\LL{2}{2} \left(\begin{array}{c} k^2 \\ k\ell \\ \ell^2 \end{array}\right).
\end{align*}
where 
\[
{\bs b}_2 = \left(\begin{array}{c}\E[I^2]\\0\\0\end{array}\right) \qquad
C\LL{2}{1}  = \left(\begin{array}{cc} 0 & {\mathbb F}[G,I]\\ p\,\E[I] &  0 \\ p(1-p) & 0\end{array}\right)  \qquad 
C\LL{2}{2}  = \left(\begin{array}{ccc} 0 & 0 & ({\mathbb E}[G])^2\\ 0 & p\,\E[G] & 0 \\ p^2 & 0 & 0 \end{array}\right) \,.
\]
Thus, as $n\rightarrow \infty$, in stationarity,
\[
\E{\bs N}_n\Hh{2} = {\bs b}_2 + C\LL{2}{1} \E{\bs N}_n\Hh{1} + C\LL{2}{2} \E{\bs N}_n\Hh{2}
\]
Rearranging terms yields
\[
\E{\bs N}_n\Hh{2} = (I - C\LL{2}{2})^{-1} \, ({\bs b}_2 + C\LL{2}{1} \E{\bs N}_n\Hh{1}) \,.
\]

(II) Similarly, we can write
\[
    \E\left[{\bs N}\Hh{3}_n\,\Big|\,X_{n-1}=k,Y_{n-1}=\ell\right] = {\bs b}_3 + C\LL{3}{1}\left(\begin{array}{c} k \\ \ell\end{array}\right)  + C\LL{3}{2} \left(\begin{array}{c} k \\ \ell\end{array}\right) + C\LL{2}{2} \left(\begin{array}{c} k^2 \\ k\ell \\ \ell^2 \end{array}\right) + C\LL{3}{3} \left(\begin{array}{c} k^3 \\ k^2\ell \\ k\ell^2 \\ \ell^3 \end{array}\right) \,,
\]
where 
\begin{align*}
&C\LL{3}{1} =
\left(
\begin{array}{cc}
0 &
3 \, \E[I^2] \, \E[G]
+ 3 \, \E[I] \, {\mathbb V}{\rm ar}[G]
+ \E[G^3]
- 3 \, \E[G] \, \E[G^2]
+ 2 (\E[G])^3
\\
p \, \E[I^2] & 0 \\
p(1-p) \, \E[I] & p(1-p) \, \E[G] \\
p(1-p)(1-2p) & 0
\end{array}
\right), \\
& C\LL{3}{2} =
\left(
\begin{array}{ccc}
0 & 0 &
3 \, \E[I] (\E[G])^2
+ 3 \, \E[G] \,  {\mathbb V}{\rm ar}[G]
\\
0 & p \bigl( 2 \, \E[I] \, \E[G] + \E[G^2] \bigr) & p (\E[G])^2 \\
p^2 \, \E[I] & 0 & 0 \\
3 p^2 (1-p) & 0 & 0
\end{array}
\right), \\
& C\LL{3}{3} =
\left(
\begin{array}{cccc}
0 & 0 & 0 & (\E[G])^3 \\
0 & 0 & 0 & 0 \\
0 & 0 & 0 & 0 \\
p^3 & 0 & 0 & 0
\end{array}
\right), \qquad \qquad
{\bs b}_3 = \left(\begin{array}{c}\E[I^3]\\0\\0\\0\end{array}\right) \,.
\end{align*}
Then, in stationarity, we have
\[
\E{\bs N}_n\Hh{3} = (I - C\LL{3}{3})^{-1} \, ({\bs b}_3 + C\LL{3}{1} \E{\bs N}_n\Hh{1} + C\LL{3}{2} \E{\bs N}_n\Hh{2}) \,.
\]

\section{Branching process with three age groups} \label{ThreeAge}
In this appendix we determine stationary moments for a branching process with three age groups, with arguments similar to those that have been used for two age groups. 
Consider the branching process
\[
X_{n+1}=I_n^{(0)}+\sum_{k=1}^2\sum_{i=1}^{Y_n^{(k)}}G_{k,i},\quad
Y_{n+1}^{(1)}\mid X_n\sim I_n^{(1)}+\mathrm{Bin}(X_n,p_1),\quad
Y_{n+1}^{(2)}\mid Y_n^{(1)}\sim I_n^{(2)}+\mathrm{Bin}(Y_n^{(1)},p_2).
\]
where $I_n^{(k)}\sim\mathrm{Poisson}(\lambda_k)$ for $k=0,1,2,$ and 
$G_{k,i}\sim\mathrm{Poisson}(\nu),$
and all random variables are mutually independent. Write
$
g_1:=\mathbb E[G]=\nu$ and $
g_2:=\mathbb E[G^2]=\nu+\nu^2,$
and define
$
Z_n:=X_n+Y_n^{(1)}+Y_n^{(2)}.$
Taking expectations and solving the stationary balance equations yields
\[
\mu:=\mathbb E[X_n]
=\frac{\lambda_0+g_1[\lambda_1(1+p_2)+\lambda_2]}
       {1-g_1p_1(1+p_2)},
\]
together with
$
\alpha:=\mathbb E[Y_n^{(1)}]=\lambda_1+p_1\mu$ and 
$\beta:=\mathbb E[Y_n^{(2)}]=\lambda_2+p_2\alpha,$
so that
\[
\mathbb E[Z_n]
=\mu+\alpha+\beta
=
\frac{\bigl[\lambda_0+g_1(\lambda_1(1+p_2)+\lambda_2)\bigr]
(1+p_1+p_1p_2)}
{1-g_1p_1(1+p_2)}
+\lambda_1(1+p_2)+\lambda_2.
\]

Next, let
\[
\boldsymbol{v}
:=
\bigl(
\mathbb E[X_n^2],\
\mathbb E[(Y_n^{(1)})^2],\
\mathbb E[(Y_n^{(2)})^2],\
\mathbb E[X_nY_n^{(1)}],\
\mathbb E[X_nY_n^{(2)}],\
\mathbb E[Y_n^{(1)}Y_n^{(2)}]
\bigr)^\top.
\]
Standard conditional moment calculations show that
$
\boldsymbol{v}
=
{C}
+{A}\boldsymbol{v},$
where
\[
{C}=
\begin{pmatrix}
\lambda_0+\lambda_0^2+(2\lambda_0g_1+g_2-g_1^2)(\alpha+\beta)\\[0.5em]
\lambda_1+\lambda_1^2+2p_1\lambda_1\mu+p_1(1-p_1)\mu\\[0.5em]
\lambda_2+\lambda_2^2+2p_2\lambda_2\alpha+p_2(1-p_2)\alpha\\[0.5em]
\lambda_0\lambda_1+\lambda_0p_1\mu+g_1\lambda_1(\alpha+\beta)\\[0.5em]
\lambda_0\lambda_2+\lambda_0p_2\alpha+g_1\lambda_2(\alpha+\beta)\\[0.5em]
\lambda_1\lambda_2+p_2\lambda_1\alpha+p_1\lambda_2\mu
\end{pmatrix},
\qquad
{A}=
\begin{pmatrix}
0 & g_1^2 & g_1^2 & 0 & 0 & 2g_1^2\\
p_1^2 & 0 & 0 & 0 & 0 & 0\\
0 & p_2^2 & 0 & 0 & 0 & 0\\
0 & 0 & 0 & g_1p_1 & g_1p_1 & 0\\
0 & g_1p_2 & 0 & 0 & 0 & g_1p_2\\
0 & 0 & 0 & p_1p_2 & 0 & 0
\end{pmatrix},
\]
so that 
$
\boldsymbol{v}
=({I}-{A})^{-1}{C}.$

The second moment and lag-one cross moment of $Z_n$ are linear functions of
$\boldsymbol{v}$. Indeed,
\[
\mathbb E[Z_n^2]
=
v_1+v_2+v_3+2v_4+2v_5+2v_6.
\]
Moreover, using conditional expectations,
\begin{align*}
\mathbb E[X_nX_{n+1}]
&=\lambda_0\mu+g_1(v_4+v_5),&
\mathbb E[X_nY_{n+1}^{(1)}]
&=\lambda_1\mu+p_1v_1,\\
\mathbb E[X_nY_{n+1}^{(2)}]
&=\lambda_2\mu+p_2v_4,&
\mathbb E[Y_n^{(1)}X_{n+1}]
&=\lambda_0\alpha+g_1(v_2+v_6),\\
\mathbb E[Y_n^{(1)}Y_{n+1}^{(1)}]
&=\lambda_1\alpha+p_1v_4,&
\mathbb E[Y_n^{(1)}Y_{n+1}^{(2)}]
&=\lambda_2\alpha+p_2v_2,\\
\mathbb E[Y_n^{(2)}X_{n+1}]
&=\lambda_0\beta+g_1(v_3+v_6),&
\mathbb E[Y_n^{(2)}Y_{n+1}^{(1)}]
&=\lambda_1\beta+p_1v_5,\\
\mathbb E[Y_n^{(2)}Y_{n+1}^{(2)}]
&=\lambda_2\beta+p_2v_6.
\end{align*}
Summing these terms gives
\[
\mathbb E[Z_nZ_{n+1}]
=
(\lambda_0+\lambda_1+\lambda_2)(\mu+\alpha+\beta)
+g_1(v_2+v_3+v_4+v_5+2v_6)
+p_1(v_1+v_4+v_5)
+p_2(v_2+v_4+v_6).
\]

We evaluate the performance of the proposed estimation procedure through a simulation study. The true parameter values are set to
\[
\lambda_0 = 0.7,\quad
\lambda_1 = 0.2,\quad
\lambda_2 = 0.1,\quad
\nu = 0.8,\quad
p_1 = p_2 = p = 0.4.
\]
We assume that $\lambda_1$ and $\lambda_2$ are known, and estimate the remaining parameters $\lambda_0$, $\nu$, and $p$. The simulation study consists of $1{,}000$ independent replications. In each replication, the system is observed over a time horizon $N = 10^5$, based on which the parameters $\lambda_0$, $\nu$, and $p$ are estimated. Each replication yields one independent realization of the estimator for each parameter.

The Q--Q plots are then constructed using the $1{,}000$ resulting independent estimates for each parameter across replications.
To assess the finite-sample distribution of the estimators, Figure \ref{fig:qq} present normal Q--Q plots of the standardized estimators for $\lambda_0$, $\nu$, and $p$, respectively. In all three cases, the sample quantiles closely track the theoretical normal quantiles, indicating that the empirical distributions of the estimators are well approximated by Gaussian distributions. Of course, alternative specifications are also possible; in particular, we considered the case in which \(p_1\), \(p_2\), and \(\nu\) are unknown and found comparable estimation performance.

\begin{figure}
\centering

\subcaptionbox{$\lambda_0$}
{\includegraphics[width=0.32\linewidth]{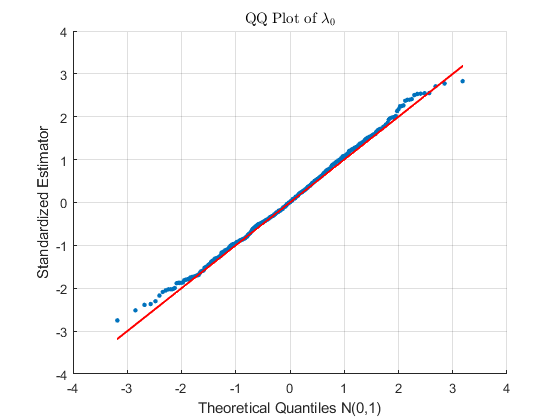}}
\subcaptionbox{$\nu$}
{\includegraphics[width=0.32\linewidth]{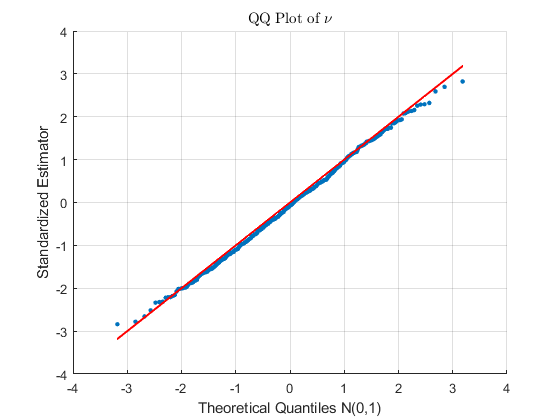}}
\subcaptionbox{$p$}
{\includegraphics[width=0.32\linewidth]{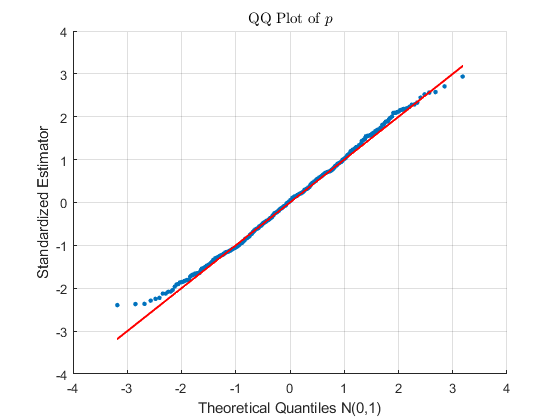}}

\caption{Normal Q--Q plots of the standardized estimators for $\lambda_0$, $\mu$, and $p$. Each estimator is standardized by subtracting its true value and dividing by its empirical standard deviation obtained from 1,000 independent replications.}
\label{fig:qq}
\end{figure}

\end{document}